\documentclass[11pt,a4paper]{article}
\usepackage[margin = 1.in]{geometry}

\usepackage{amsmath}
\usepackage{amssymb}
\usepackage{array}
\usepackage{enumerate}
\usepackage{algorithm}
\usepackage{algorithmic}
\usepackage{bm}
\usepackage[font=small,labelfont=bf]{caption}
\usepackage{subcaption}
\usepackage{changepage}
\usepackage[percent]{overpic}
\usepackage{amsfonts}
\usepackage{graphicx}
\usepackage{stackrel}
\usepackage{epstopdf}
\ifpdf
  \DeclareGraphicsExtensions{.eps,.pdf,.png,.jpg}
\else
  \DeclareGraphicsExtensions{.eps}
\fi
\usepackage{amsmath}
\usepackage[amsmath,thmmarks,hyperref]{ntheorem}
\usepackage{xr-hyper}
\usepackage{ifpdf}
\usepackage[hidelinks]{hyperref}
\usepackage[capitalize,nameinlink]{cleveref}
\usepackage[all]{hypcap}
\usepackage{graphics,graphicx}
\usepackage{xcolor}

\theoremstyle{plain}
\theoremheaderfont{\normalfont\sc}
\theorembodyfont{\normalfont\itshape}
\theoremseparator{.}
\theoremsymbol{}
\newtheorem{theorem}{Theorem}

\newcommand{\newthm}[2]{
  \theoremstyle{plain}
  \theoremheaderfont{\normalfont\sc}
  \theorembodyfont{\normalfont\itshape}
  \theoremseparator{.}
  \theoremsymbol{}
  \newtheorem{#1}[theorem]{#2}
}

\newcommand{\proofbox}{\vbox{\hrule height0.6pt\hbox{\vrule height1.3ex width0.6pt\hskip0.8ex\vrule width0.6pt}\hrule height0.6pt}}

\theoremstyle{nonumberplain}
\theoremheaderfont{\normalfont\itshape}
\theorembodyfont{\normalfont}
\theoremseparator{.}
\theoremsymbol{\proofbox}
\newtheorem{proof}{Proof}

\newthm{lemma}{Lemma}
\newthm{corollary}{Corollary}
\newthm{proposition}{Proposition}
\newthm{definition}{Definition}

\theoremstyle{nonumberplain}
\theoremheaderfont{\normalfont\itshape}
\theorembodyfont{\normalfont}
\theoremseparator{.}
\theoremsymbol{\proofbox}
\newtheorem{proofsketch}{Proof (sketch)}

\newcommand{\newremark}[2]{
  \theoremstyle{plain}
  \theoremheaderfont{\normalfont\itshape}
  \theorembodyfont{\normalfont}
  \theoremseparator{.}
  \theoremsymbol{}
  \newtheorem{#1}[theorem]{#2}
}

\newremark{remark}{Remark}
\newremark{assumption}{Assumption}
\crefname{assumption}{Assumption}{Assumption}
\creflabelformat{enumi}{(#2#1#3)}
\crefname{enumi}{}{}

\usepackage{amsopn}
\newcommand{\R}{\mathbb R}
\newcommand{\C}{\mathbb C}
\newcommand{\N}{\mathbb N}
\newcommand{\xx}{\mathbf x}
\newcommand{\uu}{\mathbf u}
\newcommand{\nn}{\mathbf n}
\newcommand{\vv}{\mathbf v}
\newcommand{\ww}{\mathbf w}
\newcommand{\pp}{\mathbf p}
\newcommand{\qq}{\mathbf q}
\newcommand{\gbf}{\mathbf g}

\newcommand{\oomega}{\bm{\omega}}
\newcommand{\mcB}{\mathcal B}
\newcommand{\mcF}{\mathcal F}
\newcommand{\mcI}{\mathcal I}
\newcommand{\mcJ}{\mathcal J}
\newcommand{\mcL}{\mathcal L}
\newcommand{\mcO}{\mathcal O}
\newcommand{\qden}{\qq_N^{\Xi_S}}
\newcommand{\Gfull}{G^{\Xi_S}}
\newcommand{\Gpart}{G_N^{\Xi_S}}
\newcommand{\Pspace}[1]{\mathbb P_{#1}\left(\C\right)}
\newcommand{\PspaceV}[2]{\mathbb P_{#1}\left(\C;#2\right)}
\newcommand{\PspaceN}[1]{\mathbb P_{#1}^\star\left(\C\right)}
\newcommand{\Pspaced}[1]{\mathbb P_{#1}\left(\C^d;\C\right)}
\newcommand{\PspacedN}[1]{\mathbb P_{#1}^\star\left(\C^d;\C\right)}
\newcommand{\ball}[2]{\mathcal B(#1,#2)}
\newcommand{\ballN}[3]{\mathcal B_{#3}(#1,#2)}
\DeclareMathOperator{\Dom}{Dom}
\DeclareMathOperator{\capac}{Cap}
\newcommand{\dual}[3]{\langle #1,#2\rangle_{#3}}
\newcommand{\dualV}[2]{\dual{#1}{#2}{V}}
\DeclareMathOperator{\dd}{d}
\newcommand{\deriv}[2]{\frac{\dd^{#2}}{\dd{#1}^{#2}}}
\newcommand{\derivi}[3]{\frac{\dd^{#3}}{\dd{#1}_{#2}^{#3}}}
\newcommand{\dmu}[1]{\deriv{\mu}{#1}}
\newcommand{\norm}[2]{\left\|#1\right\|_{#2}}
\newcommand{\normV}[1]{\left\|#1\right\|_V}
\newcommand{\normW}[1]{\left\|#1\right\|_W}
\newcommand{\abs}[1]{\left|#1\right|}
\newcommand{\abso}[1]{|#1|}
\newcommand{\normT}[1]{{\left\vert\kern-0.25ex\left\vert\kern-0.25ex\left\vert #1\right\vert\kern-0.25ex\right\vert\kern-0.25ex\right\vert}}
\newcommand{\padegen}[3]{#1_{#2}^{#3}}
\newcommand{\padefun}[2]{\padegen{#1}{N}{#2}}
\newcommand{\padex}[1]{\padefun{#1}{\Xi_S}}
\newcommand{\pade}{\padex{u}}
\newcommand{\padet}{\padefun{u}{Z_S}}
\newcommand{\padek}{\padegen{u}{N_k}{\Xi^{(k)}}}
\newcommand{\den}{Q_N^{\Xi_S}}
\newcommand{\dent}{Q_N^{Z_S}}
\newcommand{\denk}{Q_{N_k}^{\Xi^{(k)}}}
\newcommand{\jfun}{j^{\Xi_S}}
\newcommand{\jfunt}{j^{Z_S}}
\newcommand{\jfunk}{j^{\Xi^{(k)}}}

\title{Interpolatory rational model order reduction of parametric problems lacking uniform inf-sup stability\thanks{This work was funded by the Swiss National Science Foundation (SNF) through project no.~182236.}}
\author{Davide Pradovera\thanks{CSQI, EPFL, Lausanne, Switzerland (\email{davide.pradovera@epfl.ch}).}}

\newcommand{\email}[1]{\protect\href{mailto:#1}{#1}}

\crefname{section}{section}{sections}
\crefname{subsection}{subsection}{subsections}
\Crefname{section}{Section}{Sections}
\Crefname{subsection}{Subsection}{Subsections}
\Crefname{figure}{Figure}{Figures}
\crefformat{equation}{\textup{#2(#1)#3}}
\crefrangeformat{equation}{\textup{#3(#1)#4--#5(#2)#6}}
\crefmultiformat{equation}{\textup{#2(#1)#3}}{ and \textup{#2(#1)#3}}
{, \textup{#2(#1)#3}}{, and \textup{#2(#1)#3}}
\crefrangemultiformat{equation}{\textup{#3(#1)#4--#5(#2)#6}}%
{ and \textup{#3(#1)#4--#5(#2)#6}}{, \textup{#3(#1)#4--#5(#2)#6}}{, and \textup{#3(#1)#4--#5(#2)#6}}
\Crefformat{equation}{#2Equation~\textup{(#1)}#3}
\Crefrangeformat{equation}{Equations~\textup{#3(#1)#4--#5(#2)#6}}
\Crefmultiformat{equation}{Equations~\textup{#2(#1)#3}}{ and \textup{#2(#1)#3}}
{, \textup{#2(#1)#3}}{, and \textup{#2(#1)#3}}
\Crefrangemultiformat{equation}{Equations~\textup{#3(#1)#4--#5(#2)#6}}%
{ and \textup{#3(#1)#4--#5(#2)#6}}{, \textup{#3(#1)#4--#5(#2)#6}}{, and \textup{#3(#1)#4--#5(#2)#6}}
\crefdefaultlabelformat{#2\textup{#1}#3}

\ifpdf
\hypersetup{
  pdftitle={Interpolatory rational model order reduction of parametric problems lacking uniform inf-sup stability},
  pdfauthor={Davide Pradovera}
}
\fi

\newcommand{\hide}[1]{}
\newcommand{\davide}[1]{{#1}}

\begin{document}

\maketitle

\begin{abstract}
We present a technique for the approximation of a class of Hilbert space-valued maps which arise within the framework of Model Order Reduction for parametric partial differential equations, whose solution map has a meromorphic structure. Our MOR stategy consists in constructing an \emph{explicit rational} approximation based on few snapshots of the solution, in an interpolatory fashion. Under some restrictions on the structure of the original problem, we describe \emph{a priori} convergence results for our technique, hereafter called \emph{minimal rational interpolation}, which show its ability to identify the main features (e.g. resonance locations) of the target solution map. We also investigate some procedures to obtain \emph{a posteriori} error indicators, which may be employed to adapt the degree and the sampling points of the minimal rational interpolant. Finally, some numerical experiments are carried out to confirm the theoretical results and the effectiveness of our technique.
\end{abstract}

\vspace*{0.5cm}


\section{Introduction}
Mathematical models based on PDEs are employed to analyze numerically a wide array of physical, financial, and engineering-related phenomena. In many situations, such models depend on one or more parameters, either because of uncertainties or for control/design purposes, and multiple solves of the model have to be performed for different parameter values.

Often the ``naive'' approach of solving the original problem (which we will refer to as \emph{full order model}, FOM) as many times as required is not feasible, for instance because the number of evaluations is too large, or because a real-time solution is required. In such cases, model order reduction (MOR) is employed to reduce the computational effort needed for the solution of the FOM. This is achieved by constructing a surrogate model, whose solution is managed to be close to that of the FOM. The construction of the reduced model often requires a considerable computational cost, which, however, can be done \emph{offline} in a preliminary phase, whereas its evaluation at any parameter value is quite cheap, and can be carried out \emph{online} in real-time.

Many MOR techniques have been proposed for general FOMs, with the most notable and widely applied being the Reduced Basis (RB) method \cite{Chaturantabut2010,Chen2010,Chevreuil2012,Grepl2007,Hesthaven2016,Quarteroni2015,Rozza2008}. The RB method, in its simplest form, assumes that few samples of the solution of the FOM (snapshots) are enough to capture the main features of the solution manifold, i.e. the set of all solutions of the FOM for all parameter values. Exploiting this idea, a surrogate model is built by projecting the FOM onto a subspace generated by few pre-computed snapshots. The effectiveness of the RB method has been verified in many cases. However, it is possible to find some parametric problems for which projection-based MOR does not perform optimally. Such examples are often characterized by some irregularity of the FOM (e.g. \davide{nonsmooth dependence on the parameters}, non-linearity, lack of stability, etc. of the differential operator therein).

One simple case falling \davide{into} this category, which is of primary importance in our discussion, is the time-harmonic wave (Helmholtz) equation with parametric wavenumber:
\begin{equation}\label{eq:helmh}
\text{for }k\in K\subset\C,\;\text{find }u(k)\in V\;\text{s.t.}\; -\Delta u(k)-k^2u(k)\stackrel{(W)}{=}f\in W\text{,}
\end{equation}
where $V$ and $W\supset V$ are some suitable Banach function spaces. Depending on the choice of $K$ and of the boundary conditions which complement \cref{eq:helmh}, the problem above may lack inf-sup stability over $K$ \cite{Monk2003}, due to the presence of resonant wavenumbers. If one applies carelessly the RB method (in both its main versions, POD and greedy \cite{Quarteroni2015,Rozza2008}) to \cref{eq:helmh}, one may observe the appearance of spurious (``non-physical'') resonances in the surrogate model.

Discussions concerning this effect can be found in works related to MOR methods for dynamical systems (e.g. Krylov subspace methods \cite{Freund2003,Grimme1997}), where problems of similar form are quite common, due to the need for frequency-domain computations. Usually, if the number of parameters is small (often, just 1, the frequency), explicit rather than implicit approximants are considered, with the surrogate model being built by enforcing interpolation or moment matching conditions at the sample points in the parameter space, possibly in a Least Squares (LS) way. The main representative\davide{s} of this class of methods \davide{are the Loewner framework \cite{Antoulas2017,Antoulas2010,Baur2011,Beattie2014,Ionita2014} and} the vector fitting (VF) algorithm \cite{Drmac2015,Grimme1997,Grivet2016,Gustavsen1999,Lei2010,Weile1999}, which \davide{are} tailored to the specific structure of the FOM: rational functions are employed, with the objective of representing each resonance of the FOM by a pole of the approximant. This turns out to be particularly useful when the resonances of the FOM have a physical meaning, since their approximation is implicitly enclosed in the surrogate model, from which it can be obtained with small or no computational effort.

More recently, in the wake of these methods for dynamical systems, and somehow trying to profit from their main advantages, univariate LS Pad\'e approximants have been introduced and studied, both in a standard \cite{Bonizzoni2016,Bonizzoni2018} and a fast \cite{Bonizzoni2018b} version, in the context of a single parameter. Such techniques are based on multiple solves of the FOM at a single parameter value, in the same spirit as Krylov subspace methods, but yield an explicit rational approximant like VF. Actually, LS Pad\'e approximants are actually quite comparable to VF, the main differences between the methods being:
\begin{enumerate}[(i)]
\item\label{en:over} in VF, one usually must achieve a reasonably high ``sampling density'' \cite{Drmac2015}, namely use at least $2N$ samples to approximate $N$ resonances; conversely, fast LS Pad\'e approximants can get away with half as many samples;
\item the fast version of LS Pad\'e approximants relies on the high-dimensionality of the samples, namely on their linear independence, whereas VF can be applied even for scalar outputs (this is not actually a limitation of LS Pad\'e approximants, since it suffices to build a rational approximation of the solution map, and then to apply the desired linear functional to such surrogate solution);
\item\label{en:conv} the error convergence of LS Pad\'e approximants with respect to the number of samples and to the degree of the denominator is well understood for a reasonably broad class of problems \cite{Bonizzoni2016,Bonizzoni2018b}, while the convergence behavior of the VF error has, to our knowledge, yet to be analyzed thoroughly \cite{Drmac2015};
\item\label{en:centered} LS Pad\'e approximants are computed from moments of the solution at a single parameter, the explicit computation of which may feature numerically instabilities, \davide{see \cite{BonizzoniPAMM}}; instead, VF uses distributed samples in a LS-Lagrange interpolation fashion, \davide{for which ill-conditioning issues are much milder}.
\end{enumerate}

In this paper we wish to develop and analyze an extension of fast LS Pad\'e approximants which overcomes \cref{en:centered} by allowing a distributed parameter sampling, while keeping \cref{en:over} and \Cref{en:conv} intact. In particular, in \cref{sec:Preliminary results} we introduce our method, which we name \emph{minimal rational interpolation}, and we state more precisely to which kinds of FOMs our theory applies to. Then, in \cref{sec:Convergence theory} we describe the main convergence results concerning minimal rational interpolation. Afterwards, in \cref{sec:Greedy} we discuss briefly the topic of \emph{a posteriori} error/residual estimation for our technique. \davide{Possible extensions to multi-parameter problems are described in \cref{sec:conclusions_multi}.} In \cref{sec:examples} we show two numerical examples to verify the effectiveness of our method and of the presented error indicator, as well as the theoretical convergence rates. Some concluding remarks are in \cref{sec:conclusions}.

\section{Description of the method}\label{sec:Preliminary results}
Let $\left(V, \dual{\cdot}{\cdot}{V}\right)$ be a Hilbert space over $\C$, with induced norm $\normV{\cdot}$, and $K\subset\C$ be compact. Our task is to find a surrogate for a given map $u:K\to V$. Notably, we seek an approach which builds a surrogate starting only from few evaluations (snapshots) of $u$. Assuming $u$ to be the outcome of a complex computational model, such technique qualifies as \emph{non-intrusive}, in the sense that it can employ any available solver as a ``black box''.

We assume that $S>0$ samples of $u$ at the points $\Xi_S=\{\mu_j\}_{j=1}^S\subset K$ are available. In particular, under some additional regularity assumptions on $u$, we allow confluence of the sample points: if $\mu_j$ appears $K_j+1$ times in $\Xi_S$, we assume to have evaluated $u$, as well as its derivatives with respect to $\mu$ up to order $K_j$, at the point $\mu_j$.

In order to set up our approximation, it is necessary to specify how many ``resonances'' of $u$ we wish to approximate. To this aim, we consider the integer $N\in\{0,\ldots,S-1\}$, as well as the set of normalized polynomials of degree at most $N$
\begin{equation}
\PspaceN{N}=\left\{Q\in\Pspace{N},\;\normT{Q}_N=1\right\}\text{,}
\end{equation}
with \davide{$\Pspace{N}$ the space of $\C$-valued polynomials of degree up to $N$, and} $\normT{\cdot}_N$ an arbitrary (but fixed) Hilbertian norm on $\Pspace{N}$. Our aim will be to approximate $N$ resonances of $u$ by the roots of some element of $\PspaceN{N}$.

Let $I^{\Xi_S}$ be the polynomial interpolation operator\footnote{For stability reasons it may be wise to replace $I^{\Xi_S}$ with a Least Squares-type interpolator. With some effort, our results can be generalized to such case.} for $V$-valued functions based on samples at the points $\Xi_S$. Namely, given $\phi:\Xi_S\to V$, the interpolant $I^{\Xi_S}(\phi)$ is the element of \davide{the space of $V$-valued polynomials}
\begin{equation}
\PspaceV{S-1}{V}=\left\{\mu\mapsto\sum_{i=0}^{S-1}p_i\mu^i,\;p_i\in V\,\text{ for }i=0,\ldots,S-1\right\}
\end{equation}
such that $\phi(\mu')=I^{\Xi_S}(\phi)(\mu')$ for $\mu'\in\Xi_S$. In case of confluent sample points, such interpolation condition has to be interpreted in a Hermite sense, with derivatives of the interpolant matching those of the original function.

\begin{definition}[Minimal rational interpolants]\label{def:approx}
A minimal rational interpolant of $u$ of type $[S-1/N]$ with samples at $\Xi_S$ is a rational function
\begin{equation*}
\pade(\mu)=I^{\Xi_S}\left(u\den\right)(\mu)\,{\Big/}\,\den(\mu),\quad\mu\in\C,\den(\mu)\neq 0\text{,}
\end{equation*}
where the denominator $\den$ minimizes the functional
\begin{equation}\label{eq:target_functional}
\davide{\Pspace{N}\ni Q\;\mapsto\;}\jfun(Q)^2=\normV{\frac{1}{(S-1)!}\dmu{S-1}I^{\Xi_S}(uQ)}^2\text{,}
\end{equation}
over $\PspaceN{N}$.
\end{definition}

\davide{The specific expression of $\jfun$ is the natural generalization of the corresponding functional in fast LS Pad\'e approximation \cite{Bonizzoni2018b}. We will see that this form is essential to prove all our convergence results, starting from the central results described in \cref{lem:optimal_j}.}

We remark that, since $\jfun$ is convex, \davide{see \cref{eq:quadratic_form},} the existence of a minimizer $\den$ is guaranteed by the compactness of $\PspaceN{N}$. From a more practical perspective, the minimization of $\jfun$ can be carried out as follows:
\begin{enumerate}[(i)]
\item\label{item:orthoV} build a $V$-orthonormal basis $\left\{\varphi_i\right\}_{i=1}^S$ of the span of the $S$ samples $\left\{u(\mu_j)\right\}_{j=1}^S$, as well as the corresponding ``component map''
\begin{equation*}
\ww(\mu_j)=\left(\dual{u(\mu_j)}{\varphi_i}{V}\right)_{i=1}^S,\quad\text{so that}\quad u(\mu_j)=\sum_{i=1}^Sw_i(\mu_j)\varphi_i,\quad\text{for }j=1,\ldots,S\text{;}
\end{equation*}
\davide{in practice, the basis and the component map can be computed via the Gram-Schmidt procedure or by a generalized QR decomposition of the snapshots;}
\item\label{item:orthoP} fix a basis $\left\{\psi_l\right\}_{l=0}^N$ of $\Pspace{N}$, orthonormal with respect to the scalar product associated to $\normT{\cdot}_N$, and employ it to represent any arbitrary $Q\in\PspaceN{N}$ by its component vector $\qq$ (by construction, $\norm{\qq}{2}=1$); \davide{for a more detailed discussion on practical choices of the polynomial basis, we refer to \cref{sec:Convergence preliminaries};}
\item \davide{build the Gramian factor
\begin{equation*}
\Psi=\left(\frac{1}{(S-1)!}\dmu{S-1}I^{\Xi_S}(w_i\psi_l)\right)_{i=1,\ldots,S,\,l=0,\ldots,N}\in\C^{S\times(N+1)}\text{,}
\end{equation*}
which, by linearity, allows to express the target functional as a quadratic form:
\begin{align}
\jfun(Q)^2=&\frac{1}{(S-1)!^2}\dualV{\dmu{S-1}I^{\Xi_S}(uQ)}{\dmu{S-1}I^{\Xi_S}(uQ)}\nonumber\\
=&\sum_{i,i'=1}^S\sum_{l,l'=0}^N\frac{1}{(S-1)!^2}q_l\dmu{S-1}I^{\Xi_S}(w_i\psi_l)\overline{q_{l'}\dmu{S-1}I^{\Xi_S}(w_{i'}\psi_{l'})}\dualV{\varphi_i}{\varphi_{i'}}\nonumber\\
=&\sum_{i=1}^S\sum_{l,l'=0}^Nq_l\overline{q_{l'}}\,\Psi_{il}\overline{\Psi_{il'}}=\qq^H\Psi^H\Psi\qq\label{eq:quadratic_form}
\end{align}
($\overline{\cdot}$ and $\cdot^H$ denote conjugation and conjugate transposition, respectively);}
\item \davide{by SVD}, obtain the required minimizer, \davide{as the minimal right singular vector of the Gramian factor $\Psi$}.
\end{enumerate}

\davide{We remark that introducing the orthogonal bases $\left\{\varphi_i\right\}_{i=1}^S$ and $\left\{\psi_l\right\}_{l=0}^N$ is not strictly necessary for the minimization of $\jfun$. However, they allow to state the SVD problem in the usual Euclidean norm, and to reduce it to size $S$, independently of the dimension of $V$.}


\subsection{Parametric problem framework}
While the map $u$ introduced above may be a very general Hilbert-space valued function, we are often interested in approximating maps of a quite specific type. Indeed, assume that we are given a parametric problem:
\begin{equation}\label{eq:parametric_problem}
\text{for }\mu\in K\subset\C,\;\text{find }u(\mu)\in\Dom(\mcF_\mu)\;\text{s.t.}\;\mcF_\mu(u(\mu))=0\text{,}
\end{equation}
with $\mcF_\mu:V\supset\Dom(\mcF_\mu)\to V$ a family of densely defined operators, depending continuously on $\mu$. Provided the set $K^\star\subset K$ of values of $\mu$, for which a unique $u(\mu)$ solving \cref{eq:parametric_problem} exists, is non-empty, we focus on the \emph{solution map} associated to \cref{eq:parametric_problem}:
\begin{equation}\label{eq:solution_map}
\begin{aligned}
u:K^\star & \to V\\
\mu & \mapsto u(\mu)\text{ solution of }\cref{eq:parametric_problem}\text{.}
\end{aligned}
\end{equation}

\davide{In the following, we restrict our analysis to a special class of solution maps, namely meromorphic $V$-valued functions over $\C$ with simple poles:
\begin{equation}\label{eq:map_decomposition}
u(\mu)=\sum_{\lambda\in\Lambda}\frac{v_\lambda}{\lambda-\mu},\quad\mu\in\C\text{.}
\end{equation}
We denote by $\{v_\lambda\}_{\lambda\in\Lambda}\subset V$ the set of \emph{residues}, which we assume to be $V$-square-summable: $\sum_{\lambda\in\Lambda}\normV{v_\lambda}^2<\infty$. Moreover, we assume that the set $\Lambda$ of \emph{poles}, or \emph{resonances}, is countable, with no finite limit point. If $\Lambda$ is not finite, \cref{eq:map_decomposition} has to be understood in the sense of $V$-convergence, for an arbitrary ordering of $\Lambda$. Without loss of generality, we will assume that $\Lambda$ is unbounded: if this is not the case, one can always extend $\Lambda$ by adding (countably) infinitely many fictitious poles at $\infty$ with arbitrary corresponding residues.

The bulk of our results will be obtained under the additional assumption that the residues form an orthogonal family, i.e. that $\dualV{v_\lambda}{v_{\lambda'}}=0$ whenever $\lambda\neq\lambda'$. This allows considerable simplifications in our derivation: in particular, orthogonal residues have the great advantage of introducing no ``correlation'' (in the $V$-topology) between components of the solution corresponding to different resonances. Still, many of our conclusions can be generalized to the non-orthogonal case: we defer a discussion of such extensions to \cref{sec:nonortho}.

To motivate our interest in this class of functions, we give here some examples of operator families $\{\mcF_\mu\}_{\mu\in K}$ which lead to meromorphic solution maps; for their approximation, it is reasonable to employ rational functions, as minimal rational interpolants do. However, we remark that the theory that we develop might not necessarily apply to all the cases which we list here.
\begin{itemize}
\item Frequency-domain dynamical systems: $\mcF_\mu(u)=(A-\mu E)u+B$, with $A,E\in\C^{n\times n}$, and $B\in V=\C^{n\times n_I}$ (one column per input). By employing a determinantal representation \cite{Baker1996}, one can easily show that the solution map is a rational function; in particular, if the pencil $(A,E)$ admits $n$ (not necessarily finite) eigenvalues counting multiplicity, then $u$ is of the form \cref{eq:map_decomposition} \cite{Demmel2000}. 
Depending on the spectral properties of $A$, one may be able to ensure residue orthogonality: for instance, it suffices for $E^\dagger A$ to be normal, where ${}^\dagger$ denotes the Moore-Penrose pseudo-inverse.
\item Polynomial eigenvalue-like problems: $\mcF_\mu(u)=A(\mu)u+B$, where $A(\mu)\in\C^{n\times n}$ is a polynomial in $\mu$, and $B\in V=\C^{n\times n_I}$. Using augmentation-based approaches \cite{Demmel2000,Guillaume1999}, one can cast the problem as a frequency-domain dynamical system of size $nM$, where $M$ is the degree of $A$. This implies that the solution map of the original problem (of size $n$) is a rational function; in order to check whether it is actually of the form \cref{eq:map_decomposition}, one has to investigate the spectral properties of the augmented $(nM)\times(nM)$ problem.
\item Helmholtz-type problems: $\mcF_\mu(u)=(\mcL-\mu\mcI)u-f$, where $\mcL:V\supset\Dom(\mcL)\to V$ is densely defined and bijective, with compact and normal inverse, and $f\in V$ (e.g. $V=L^2(\Omega)$, $\mcL=-\Delta$, and $\Dom(\mcL)=\Delta^{-1}(V)$). The solution map is of the form \cref{eq:map_decomposition}, and the normality of the resolvent $\mcL^{-1}$ ensures orthogonal residues \cite{Bonizzoni2018b}.
\item Analytic compact perturbations of bounded bijective operators: $\mcF_\mu(u)=(\mcJ-\mcL(\mu))u-f$, where $\mcJ:V\to V$ is bounded and bijective, $\mcL(\mu):V\to V$ is compact for all $\mu\in\C$, and analytic with respect to $\mu$, and $f\in V$; a proof of meromorphy is in \cite{Steinberg1968}. Unfortunately, the form of the solution map is, in general, not \eqref{eq:map_decomposition}: for instance, it cannot be guaranteed that the poles are simple, i.e. a $\lambda$-dependent exponent may appear in the denominator of \eqref{eq:map_decomposition}. 
\end{itemize}}

\section{Convergence theory}\label{sec:Convergence theory}
In this section, we investigate the convergence properties of minimal rational interpolants as the number of samples $S$ increases. To this aim, it becomes necessary to introduce some assumptions on the asymptotic properties of the sample set $\Xi_S$. In order to streamline the derivation, we choose to follow a potential theory-inspired approach.

\subsection{Preliminaries for convergence theory}\label{sec:Convergence preliminaries}
Fix $K\subset\C$ either as the closure of the finite region bounded by a Jordan curve or a line segment with positive length. There exist \cite{Walsh1969} a (unique) positive real number $\capac(K)$ and a continuous bijective conformal mapping
\begin{equation*}
\phi_K:\overline{\C\setminus K}\to\C\setminus\ball{0}{\capac(K)}=\{z\in\C,\abs{z}\geq\capac(K)\}\text{,}
\end{equation*}
holomorphic outside $K$, such that $\abs{\phi_K}=\capac(K)$ over $\partial K$ and $\lim_{\mu\to\infty}\abs{\phi_K'(\mu)}=1$.

The value $\capac(K)$ is the \emph{logarithmic capacity} of $K$. It is a powerful concept in complex analysis, algebraic geometry, and approximation theory, and we refer to the monograph \cite{Hille65} for a thorough introduction to the subject. In its more general framework, logarithmic capacity is well defined for any compact set in the complex plane, and coincides with the definition above under our assumptions.

Among its properties we remark some bounds for other metrics over complex sets: for any compact set $A\subset\C$ we have
\begin{equation}\label{eq:capacity_bounds}
\sqrt{\frac{\ell(A)}{\pi}}\leq\capac(A)\leq\frac{1}{2}\max_{z,w\in A}\abs{z-w}\text{,}
\end{equation}
with $\ell$ denoting the Lebesgue measure over $\C$. We remark that zero-measure sets are not guaranteed to have capacity zero: for instance, $\capac([0,4])=1$.

A crucial quantity for stating convergence results in approximation theory is the Green's potential, which is closely related to the complex magnitude of $\phi_K$:
\begin{equation}\label{eq:green_potential}
\Phi_K(\mu)=
\begin{cases}
\abs{\phi_K(\mu)}\quad&\text{if }\mu\notin K,\\
\capac(K)\quad&\text{if }\mu\in K
\end{cases}
\text{.}
\end{equation}
In particular, we use the Green's potential to induce an ordering in the set of poles $\Lambda=\{\lambda_j\}_{j=1,2,\ldots}$:
\begin{equation}\label{eq:pole_ordering}
\Phi_K(\lambda_1)\leq\Phi_K(\lambda_2)\leq\ldots\text{.}
\end{equation}
Moreover, we require that $\Phi_K(\lambda_{N+1})>\capac(K)$, i.e. that the number of poles within $K$ is at most $N$, see \cref{eq:green_potential} and \cref{eq:pole_ordering}. Indeed, we will show in the following that minimal rational interpolants fail to differentiate between poles of $u$ within $K$; hence, without the assumption $\Phi_K(\lambda_{N+1})>\capac(K)$, the ``unidentifiable'' poles would distort the approximant and, in general, prevent convergence to the target solution map.

Now, for a given finite point set $\Xi\subset\C$, we define the corresponding nodal polynomial as $\omega^\Xi(\mu)=\prod_{\mu'\in\Xi}\left(\mu-\mu'\right)$. This allows us to state the result which motivates our interest in Green's potentials.

\begin{theorem}[Fej\'er points \cite{Walsh1969}]
Let $K$ be either the closure of the finite region bounded by a Jordan curve or a line segment with positive length. There exists a sequence of sample sets $\big(\Xi_S^{K}\big)_{S=1,2,\ldots}$, whose elements are referred to as the Fej\'er points for $K$, such that
\begin{gather}
\limsup_{S\to\infty}\abs{\omega^{\Xi_S^K}}^{1/S}\leq\Phi_K\quad\text{uniformly over compact subsets of }\C\text{,}\label{eq:nodal_conv}\\
\lim_{S\to\infty}\abs{\omega^{\Xi_S^K}}^{1/S}=\Phi_K\quad\text{uniformly over compact subsets of }\C\setminus\partial K\text{.}\label{eq:nodal_conv_out}
\end{gather}
\end{theorem}
\davide{This means that the Green's potential provides an asymptotic upperbound for the magnitude of the nodal polynomial, as long as the sample points are chosen in an optimal way}. From here onward we assume that, given $S\in\N$, the sample set $\Xi_S$ coincides with some set of Fej\'er points $\Xi_S^K$, allowing us to drop the superscript.

\begin{remark}
\davide{
Fej\'er points are explicitly known only for very simple sets $K$; this is the case, for instance, in two situations which are very common in applications: the Fej\'er points of the unit disk are the roots of unity, and for a line segment they are the Chebyshev nodes \cite{Walsh1969}. However, our theory applies even for more general sampling strategies. The technical requirement is the following: there exists $\Phi:\C\to\R$ such that
\begin{itemize}
\item $\sup_K\Phi\leq\inf_{\C\setminus K}\Phi$,
\item $\limsup_{S\to\infty}\abs{\omega^{\Xi_S}}^{1/S}\leq\Phi$ uniformly over $K$,
\item $\lim_{S\to\infty}\abs{\omega^{\Xi_S}}^{1/S}=\Phi$ uniformly over compact subsets of $\C\setminus K$.
\end{itemize}
If this is the case, then one can replace the Green's potential by $\Phi$ in the entirety of our discussion; in particular, $\Phi$ should induce the pole ordering, cf. \cref{eq:pole_ordering}.

For instance, one could take any sequence which approximates asymptotically the Fej\'er points \cite{Bos2008}. Another interesting case is that of repeated points, employing Lagrange-Hermite interpolation: given $\{\mu_i^0\}_{i=1}^L\subset\C$, one defines the sample points as $\mu_{nL+i}=\mu_i^0$ for all $n=0,1,\ldots$ and $i=1,\ldots,L$. Then, as long as $K$ is defined as the discrete set $K=\{\mu_i^0\}_{i=1}^L$, the three conditions above hold with $\Phi(\mu)=\prod_{i=1}^L\abs{\mu-\mu_i^0}^{1/L}$. In particular, this allows to analyze fast LS Pad\'e approximants \cite{Bonizzoni2018b} as a special case of minimal rational interpolants, with $L=1$ and $\Phi(\mu)=\abs{\mu-\mu_1^0}$.
}
\end{remark}

A second class of considerations involves the denominator space $\PspaceN{N}$, whose definition hinges upon the choice of the norm $\normT{\cdot}_N$. Depending on the type of result we wish to prove, we may need to introduce the following assumption on the scaling of $\normT{\cdot}_N$ with respect to $N$.

\begin{assumption}\label{as:with_scale}
There exist $\mu_0\in\C$ and three positive constants $R_{\mu_0}$, $c_{\mu_0}$, and $C_{\mu_0}$, all independent of $N$, such that
\begin{equation*}
c_{\mu_0}\prod_{j=1}^D\abs{\mu_0-z_j}\leq\normT{\prod_{j=1}^D\left(\,\cdot-z_j\right)}_N\leq C_{\mu_0}\prod_{j=1}^D\left(R_{\mu_0}+\abs{\mu_0-z_j}\right)\text{,}
\end{equation*}
for all $N\in\N$ and $\{z_j\}_{j=1}^D\in\C^D$, with $0\leq D\leq N$.
\end{assumption}

In \cite{Bonizzoni2018b} it was shown that \cref{as:with_scale} holds (with $R_{\mu_0}=c_{\mu_0}=C_{\mu_0}=1$) for the ``monomial'' norm $\normT{\cdot}_N=\normT{\cdot}_{\mu_0,N}$ defined as
\begin{equation}\label{eq:monomial_norm}
\normT{Q}_{\mu_0,N}=\sqrt{\sum_{j=0}^N\abs{q_j}^2},\quad\text{where}\quad Q=\sum_{j=0}^Nq_j\left(\,\cdot-\mu_0\right)^j\text{.}
\end{equation}
We describe in \cref{ap:with_scale} a broader framework where \cref{as:with_scale} can still be shown to hold, with the (often ill-conditioned) monomials in \cref{eq:monomial_norm} being replaced by a more general polynomial basis.

Many of our results do not rely on \cref{as:with_scale}, but exploit the following weaker result.
\begin{lemma}\label{as:no_scale}
For all $\mu_0\in\C$, there exist positive constants $R_{\mu_0}$ (independent of $N$), $c_{\mu_0,N}$, and $C_{\mu_0,N}$, such that
\begin{equation*}
c_{\mu_0,N}\prod_{j=1}^D\abs{\mu_0-z_j}\leq\normT{\prod_{j=1}^D\left(\,\cdot-z_j\right)}_N\leq C_{\mu_0,N}\prod_{j=1}^D\left(R_{\mu_0}+\abs{\mu_0-z_j}\right)\text{,}
\end{equation*}
for all $N\in\N$ and $\{z_j\}_{j=1}^D\in\C^D$, with $0\leq D\leq N$. 
\end{lemma}

\begin{proof}
As described above, $\normT{\cdot}_N=\normT{\cdot}_{\mu_0,N}$ satisfies \cref{as:no_scale}. The result follows by the equivalence of all norms on a finite-dimensional space.
\end{proof}

\subsection{Convergence of minimal rational interpolant poles}\label{sec:Pole convergence}
In the present section we prove that the denominator $\den$ described in \cref{def:approx} performs well in identifying the approximate position of some of the poles $\Lambda$ of the solution map, among which those within the parameter domain $K$. As a preliminary step, it is useful to exploit the definition of $\den$ to derive a bound for the minimal value of $\jfun$ over $\PspaceN{N}$.
\begin{lemma}\label{lem:optimal_j}
Let $Q\in\Pspace{N}$. Then
\begin{equation*}
\jfun(Q)^2=\sum_{\lambda\in\Lambda}\frac{\normV{v_\lambda}^2}{\abs{\omega^{\Xi_S}(\lambda)}^2}\abs{Q(\lambda)}^2\text{.}
\end{equation*}
Moreover, for any subset $\Lambda_0$ of $\Lambda$ with cardinality at most $N$, we have
\begin{equation}\label{eq:jfun_optimum_weird}
\jfun(\den)\leq\frac{\left(\sum_{\lambda\in\Lambda\setminus\Lambda_0}\normV{v_\lambda}^2\right)^{1/2}}{\normT{\omega^{\Lambda_0}}_N}\sup_{\Lambda\setminus\Lambda_0}\abs{\frac{\omega^{\Lambda_0}}{\omega^{\Xi_S}}}\text{.}
\end{equation}
In particular, for any $\varepsilon>0$, there exists $S_{\varepsilon,N}$ such that
\begin{equation}\label{eq:jfun_optimum}
\jfun(\den)\leq (1+\varepsilon)^S\frac{\left(\sum_{j>N}\normV{v_{\lambda_j}}^2\right)^{1/2}}{\normT{\prod_{j=1}^N\left(\ \cdot-\lambda_j\right)}_N}\frac{\prod_{j=1}^N\abs{\lambda_{N+1}-\lambda_j}}{\Phi_K(\lambda_{N+1})^S}\quad\text{for }S>S_{\varepsilon,N}\text{,}
\end{equation}
provided $\Phi_K(\lambda_{N+1})<\Phi_K(\lambda_{N+2})$. (If this is not true, there exists a reordering of the poles, still satisfying \cref{eq:pole_ordering}, for which \cref{eq:jfun_optimum} holds.)
\end{lemma}

\begin{proof}
The interpolation operator can be expressed in barycentric coordinates \davide{\cite{Berrut2004}} as
\begin{equation}\label{eq:interpolant_explicit}
\davide{I^{\Xi_S}(\phi)(\mu)=\sum_{\mu'\in\Xi_S}\frac{\phi(\mu')\,\omega^{\Xi_S}(\mu)}{(\mu-\mu')\left.\omega^{\Xi_S}\right.'(\mu')}\text{.}}
\end{equation}
Accordingly, since $Q$ is interpolated exactly,
\begin{align*}
\frac{1}{(S-1)!}\dmu{S-1}I^{\Xi_S}(uQ)=&\sum_{\mu'\in\Xi_S}\frac{u(\mu')Q(\mu')}{\left.\omega^{\Xi_S}\right.'(\mu')}=\sum_{\lambda\in\Lambda}\sum_{\mu'\in\Xi_S}\frac{v_\lambda Q(\mu')}{(\lambda-\mu')\left.\omega^{\Xi_S}\right.'(\mu')}\\
=&\sum_{\lambda\in\Lambda}\frac{v_\lambda}{\omega^{\Xi_S}(\lambda)}I^{\Xi_S}(Q)(\lambda)=\sum_{\lambda\in\Lambda}\frac{v_\lambda}{\omega^{\Xi_S}(\lambda)}Q(\lambda)\text{.}
\end{align*}
The first claim follows by orthonormality of $\{v_\lambda\}_{\lambda\in\Lambda}$. From here, \davide{we remark that, by construction, $\jfun(\den)\leq\jfun\left(Q\right)$ for all $Q\in\PspaceN{N}$, so that, by choosing $Q=\omega^{\Lambda_0}/\normT{\omega^{\Lambda_0}}_N$, we obtain}
\begin{equation}\label{eq:optimal_j_bound_middle}
\jfun(\den)^2\leq\jfun\left(\frac{\omega^{\Lambda_0}}{\normT{\omega^{\Lambda_0}}_N}\right)^2=\sum_{\lambda\in\Lambda\setminus\Lambda_0}\frac{\normV{v_\lambda}^2}{\abs{\omega^{\Xi_S}(\lambda)}^2}\frac{\abs{\omega^{\Lambda_0}(\lambda)}^2}{\normT{\omega^{\Lambda_0}}_N^2}\text{,}
\end{equation}
from which \cref{eq:jfun_optimum_weird} follows.

Now, let $\varepsilon>0$ and $N$ be fixed, and define $\Lambda_0=\{\lambda_j\}_{j=1}^N$. It remains to find a bound for the supremum of \davide{$r^{\Lambda_0,\Xi_S}(\lambda)=|\omega^{\Lambda_0}(\lambda)/\omega^{\Xi_S}(\lambda)|$ over $\lambda\in\Lambda\setminus\Lambda_0$}. To this aim, let $\lambda'_S$ maximize $r^{\Lambda_0,\Xi_S}$ over $\Lambda\setminus\Lambda_0$ (for all $S$, the supremum must be attained since $r^{\Lambda_0,\Xi_S}(\lambda')$ converges to 0 as $\abs{\lambda'}$ diverges to infinity). The set of maximizers $\{\lambda'_S\}_{S>N}$ must be bounded, due to the polynomial degree of $\omega^{\Lambda_0}$ being fixed, and smaller than that of $\omega^{\Xi_S}$. Thus, by \cref{eq:nodal_conv_out} there exists $S'=S'(\varepsilon,N,\Lambda,K)$ such that, for $S>S'$,
\begin{equation*}
\sup_{\Lambda\setminus\Lambda_0}r^{\Lambda_0,\Xi_S}=\sup_{\{\lambda'_S\}_{S>N}}r^{\Lambda_0,\Xi_S}\leq(1+\varepsilon)^S\sup_{\{\lambda'_S\}_{S>N}}\frac{|\omega^{\Lambda_0}|}{\Phi_K^S}\leq(1+\varepsilon)^S\sup_{\Lambda\setminus\Lambda_0}\frac{|\omega^{\Lambda_0}|}{\Phi_K^S}\text{.}
\end{equation*}
Let $\Phi_K(\lambda_{N+1})=\Phi_K(\lambda_{N+\nu})<\Phi_K(\lambda_{N+\nu+1})$ for some $\nu\geq 1$, and assume that $\lambda_{N+1}$ maximizes \davide{$|\omega^{\Lambda_0}(\lambda)|$ over $\lambda\in\{\lambda_{N+k}\}_{k=1}^\nu$}. Then, since $|\omega^{\Lambda_0}|$ is independent of $S$, the last supremum is attained at $\lambda_{N+1}$, provided $S$ is large enough. This yields \cref{eq:jfun_optimum}.
\end{proof}

Before proving convergence of the roots of the approximate denominator to poles of the solution map, we show as a preliminary result that $\den$ takes small values close to some of the elements of $\Lambda$.
\begin{lemma}\label{lem:den_N_lambda}
Let $N$ be fixed. For any $\varepsilon>0$, there exists $S_{\varepsilon,N}$ such that
\begin{equation}\label{eq:den_N_lambda}
\abs{\den\left(\lambda_j\right)}\leq C_j(1+\varepsilon)^{4S}\left(\frac{\Phi_K(\lambda_j)}{\Phi_K(\lambda_{N+1})}\right)^{2S}\quad\text{for }S>S_{\varepsilon,N}\text{,}
\end{equation}
for $j=1,\ldots,N$, with $C_j$ independent of $S$ and $\varepsilon$.
\end{lemma}

\begin{proof}
Through some isometry, we identify $\Pspace{N}$ with $\C^{N+1}$ (endowed with the Euclidean norm), so that, in particular, $\PspaceN{N}$ corresponds to the boundary of the unit sphere $\partial\ballN{\bm{0}}{1}{N+1}\subset\C^{N+1}$.

By duality, there exists a family $\{\oomega_\lambda\}_{\lambda\in\Lambda}\subset\C^{N+1}$ such that
\begin{equation*}
\Pspace{N}\ni Q\leftrightsquigarrow\qq\in\C^{N+1}\quad\text{implies}\quad\oomega_\lambda^H\qq=Q(\lambda)\;\text{for }\lambda\in\Lambda\text{.}
\end{equation*}
Let $\qden\in\partial\ballN{\bm{0}}{1}{N+1}$ be the identification of $\den$, and consider the Hermitian matrices $\Gfull,\Gpart\in\C^{(N+1)\times(N+1)}$ defined as
\begin{equation*}
\Gfull=\sum_{\lambda\in\Lambda}\frac{\normV{v_\lambda}^2}{\abs{\omega^{\Xi_S}(\lambda)}^2}\oomega_\lambda^{}\oomega_\lambda^H\quad\text{and}\quad\Gpart=\sum_{j=1}^N\frac{\normV{v_{\lambda_j}}^2}{\abs{\omega^{\Xi_S}(\lambda_j)}^2}\oomega_{\lambda_j}^{}\oomega_{\lambda_j}^H\text{,}
\end{equation*}

From here onward, the proof resembles quite closely that of Lemma~\davide{5.4} in \cite{Bonizzoni2018b}, being essentially based on deriving a bound for the smallest singular value of $\Gfull$ and on finding a rank-$N$ decomposition of $\Gpart$. The final result that can be obtained is of the form: for $j=1,\ldots,N$, for large $S$,
\begin{equation*}
\abs{\den\left(\lambda_j\right)}=\abs{\oomega_{\lambda_j}^H\qden}\leq C_j(1+\varepsilon)^{2S}\frac{\abs{\omega^{\Xi_S}(\lambda_j)}^2}{\Phi_K(\lambda_{N+1})^{2S}}\text{,}
\end{equation*}
with $C_j$ independent of $S$ and $\varepsilon$. From here it suffices to observe that $\abs{\omega^{\Xi_S}(\lambda_j)}$ is bounded by $\left((1+\varepsilon)\Phi_K(\lambda_j)\right)^S$, thanks to \cref{eq:nodal_conv}.
\end{proof}

Finally, we are able to show that the poles of minimal rational approximants converge, as $S$ increases, to some of the elements of $\Lambda$, namely the $N$ poles which are the ``closest'' according to \cref{eq:pole_ordering}.
\begin{theorem}\label{th:poles_N}
Let $N$ be fixed, and denote by $\{\lambda_j^{\Xi_S}\}_{j=1}^N$ the roots of $\den$ (in case of deficient degree, we assume the missing roots to be $\infty$). For any $\varepsilon>0$, there exists $S_{\varepsilon,N}$ such that
\begin{equation}\label{eq:poles_N}
\min_{i=1,\ldots,N}\abs{\lambda_i^{\Xi_S}-\lambda_j}\leq C_j(1+\varepsilon)^{4S}\left(\frac{\Phi_K(\lambda_j)}{\Phi_K(\lambda_{N+1})}\right)^{2S}\quad\text{for }S>S_{\varepsilon,N}\text{,}
\end{equation}
for $j=1,\ldots,N$, with $C_j$ independent of $S$ and $\varepsilon$.
\end{theorem}

\begin{proof}
We start by observing that, by the normalization of $\den$, it must hold
\begin{equation}\label{eq:poles_N_den_norm}
\abs{\den(\lambda)}=\frac{\prod_{j=1}^S\abs{\lambda-\lambda_j^{\Xi_S}}}{\normT{\prod_{j=1}^S\left(\,\cdot-\lambda_j^{\Xi_S}\right)}_N}\geq\frac{1}{C_{\mu_0,N}}\prod_{j=1}^{N}\frac{\abs{\lambda-\lambda_j^{\Xi_S}}}{R_{\mu_0}+\abs{\lambda-\mu_0}+\abs{\lambda-\lambda_j^{\Xi_S}}}\\
\end{equation}
for any arbitrary $\mu_0\in\C$, see \cref{as:no_scale}.

From here it suffices to apply the same proof as for Theorem~5.5 in \cite{Bonizzoni2018b}, with \cref{eq:den_N_lambda} replacing the corresponding bound for fast LS-Pad\'e denominators. \davide{The main steps are: (i) use \cref{eq:den_N_lambda} and \cref{eq:poles_N_den_norm} to show that, for each $j=1,\ldots,N$, there is a sequence of approximate poles $\{\lambda_{i(S)}^{\Xi_S}\}_{S>S_{\varepsilon,N}}$ which converges to $\lambda_j$; (ii) exploiting the fact that the poles of $u$ are simple, conclude that $\den$ cannot have, for large $S$, two roots too close to each other; (iii) apply \cref{eq:den_N_lambda} locally to derive \cref{eq:poles_N}.}
\end{proof}

Going back to \davide{our} original problem formulation, let $0\leq\nu\leq N$ be such that $\capac(K)=\Phi_K(\lambda_\nu)<\Phi_K(\lambda_{\nu+1})$. \cref{th:poles_N} shows that, for fixed $N$ and large enough $S$, minimal rational interpolants approximate well the $\nu$ poles $\{\lambda_j\}_{j=1}^\nu$ within $K$, with (exactly) one root of $\den$ converging to each of them at rate
\begin{equation}\label{eq:poles_N_inside}
\mcO\left(\left(\capac(K)\,\big/\,\Phi_K(\lambda_{N+1})\right)^{2S}\right)\text{.}
\end{equation}

\subsubsection{Pole convergence with respect to denominator degree}\label{sec:Pole N convergence}
The statements of \cref{lem:den_N_lambda} and \cref{th:poles_N} may seem somewhat limited due to the fact that the denominator degree $N$ enters the convergence bounds in a complex way. This makes it impossible to determine how the pole approximation error behaves if $N$ is increased together with $S$, namely if it diverges to infinity. In this section we try to solve this shortcoming. The main result we prove is the following.
\begin{theorem}\label{th:poles_notN}
Take two diverging sequences of integers $(N_k)_{k\in\N}$ and $(S_k)_{k\in\N}$, increasing and strictly increasing respectively, such that $N_k<S_k$ for all $k$. Let $\Xi^{(k)}$ be the shorthand for $\Xi_{S_k}$, so that $\smash{\denk}$ is the denominator of the $\smash{[S_k-1/N_k]}$ minimal rational interpolant computed from samples of $u$ at $\Xi_{S_k}$. We denote by $\smash{\{\lambda_j^{(k)}\}_{j=1}^{N_k}}$ the roots of $\smash{\denk}$. Also, let \cref{as:with_scale} be satisfied. For all $j=1,2,\ldots$, we have
\begin{equation*}
\lim_{k\to\infty}\min_{i=1,\ldots,N_k}\abs{\lambda_i^{(k)}-\lambda_j}=0\text{.}
\end{equation*}
\end{theorem}

\begin{proof}
Since we are interested in asymptotic properties of the poles, in the present proof we employ a slightly different ordering of the resonances $\Lambda=\{\sigma_j\}_{j=1,2,\ldots}$, such that
\begin{equation}\label{eq:pole_ordering_mu0}
\abs{\sigma_1-\mu_0}\leq\abs{\sigma_2-\mu_0}\leq\ldots\text{.}
\end{equation}

Let $j\in\N$ be arbitrary. Thanks to \cref{as:with_scale}, we can obtain the lower bound
\begin{align*}
\abs{\denk(\sigma_j)}=\frac{\prod_{i=1}^{N_k}\abs{\sigma_j-\lambda_i^{(k)}}}{\normT{\prod_{i=1}^{N_k}\left(\,\cdot-\lambda_i^{(k)}\right)}_{N_k}}\geq&\frac{1}{C_{\mu_0}}\prod_{i=1}^{N_k}\frac{\abs{\sigma_j-\lambda_i^{(k)}}}{R_{\mu_0}+\abs{\sigma_j-\mu_0}+\abs{\sigma_j-\lambda_i^{(k)}}}\\
=&\frac{1}{C_{\mu_0}}\prod_{i=1}^{N_k}\varphi_j\left(\abs{\sigma_j-\lambda_i^{(k)}}\right)\text{,}
\end{align*}
with $\varphi_j=\cdot/(R_{\mu_0}+\abs{\sigma_j-\mu_0}+\cdot)$ non-negative and strictly increasing over $\R^+$. On the other hand, due to \cref{as:with_scale} and \cref{lem:optimal_j} with $\Lambda_0=\{\sigma_i\}_{i=1}^{N_k}$, one has
\begin{align*}
\abs{\denk(\sigma_j)}\leq&\frac{\abs{\omega^{\Xi^{(k)}}(\sigma_j)}}{\normV{v_{\sigma_j}}}\jfunk(\denk)\\
\leq &\frac{\abs{\omega^{\Xi^{(k)}}(\sigma_j)}}{\normV{v_{\sigma_j}}}\frac{\left(\sum_{\lambda\in\Lambda\setminus\Lambda_0}\normV{v_\lambda}^2\right)^{1/2}}{\normT{\prod_{i=1}^{N_k}\left(\ \cdot-\sigma_i\right)}_N}\sup_{l>N_k}\frac{\prod_{i=1}^{N_k}\abs{\sigma_l-\sigma_i}}{\abs{\omega^{\Xi^{(k)}}(\sigma_l)}}\\
\leq &\frac{\left(\sum_{\lambda\in\Lambda}\normV{v_\lambda}^2\right)^{1/2}}{c_{\mu_0}\normV{v_{\sigma_j}}}\sup_{l>N_k}\left(\abs{\frac{\omega^{\Xi^{(k)}}(\sigma_j)}{\omega^{\Xi^{(k)}}(\sigma_l)}}\prod_{i=1}^{N_k}\abs{\frac{\sigma_l-\sigma_i}{\mu_0-\sigma_i}}\right)\text{.}
\end{align*}
The two bounds above can be combined with the monotonicity of $\varphi_j$ to obtain
\begin{equation}\label{eq:poles_notN_bound}
\varphi_j\left(\min_{i=1,\ldots,N_k}\abs{\sigma_j-\lambda_i^{(k)}}\right)\leq\left(C_j'\sup_{l>N_k}\left(\abs{\frac{\omega^{\Xi^{(k)}}(\sigma_j)}{\omega^{\Xi^{(k)}}(\sigma_l)}}\prod_{i=1}^{N_k}\abs{\frac{\sigma_l-\sigma_i}{\mu_0-\sigma_i}}\right)\right)^{1/N_k}\text{,}
\end{equation}
with $C_j'=C_{\mu_0}(\sum_{\lambda\in\Lambda}\normV{v_\lambda}^2)^{1/2}/(c_{\mu_0}\normV{v_{\sigma_j}})$.

Thanks to the bijectivity of $\varphi_j$, which maps 0 to itself, the claim follows if we can show that the right hand side of \cref{eq:poles_notN_bound} converges to 0 as $k$ increases. To this aim, a first step is to remove the supremum, whose argument can be simplified by introducing the bounds
\begin{equation*}
\prod_{i=1}^{N_k}\abs{\frac{\sigma_l-\sigma_i}{\mu_0-\sigma_i}}\leq\prod_{i=1}^{N_k}\frac{\abs{\mu_0-\sigma_l}+\abs{\mu_0-\sigma_i}}{\abs{\mu_0-\sigma_i}}\leq\prod_{i=1}^{N_k}\frac{2\abs{\mu_0-\sigma_l}}{\abs{\mu_0-\sigma_i}}=\frac{2^{N_k}\abs{\mu_0-\sigma_l}^{N_k}}{\prod_{i=1}^{N_k}\abs{\mu_0-\sigma_i}}
\end{equation*}
and
\begin{align*}
\abs{\frac{\omega^{\Xi^{(k)}}(\sigma_j)}{\omega^{\Xi^{(k)}}(\sigma_l)}}=\prod_{\mu'\in\Xi^{(k)}}\abs{\frac{\mu'-\sigma_j}{\mu'-\sigma_l}}\leq\prod_{\mu'\in\Xi^{(k)}}\frac{\abs{\mu_0-\sigma_j}+\abs{\mu_0-\mu'}}{\abs{\abs{\mu_0-\sigma_l}-\abs{\mu_0-\mu'}}}\text{,}
\end{align*}
where the term $\abs{\mu_0-\mu'}$ is bounded from above by $R_{\mu_0,K}=\max_K\abs{\,\cdot-\mu_0}$.

Now, let $k$ be large enough, so that $\abs{\mu_0-\sigma_{N_k+1}}>R_{\mu_0,K}$. Then the absolute value in the last bound can be omitted, and it must hold
\begin{equation}\label{eq:poles_notN_bound_sup}
\abs{\frac{\omega^{\Xi^{(k)}}(\sigma_j)}{\omega^{\Xi^{(k)}}(\sigma_l)}}\prod_{i=1}^{N_k}\abs{\frac{\sigma_l-\sigma_i}{\mu_0-\sigma_i}}\leq(C_{l,k,j,\mu_0})^{N_k}\prod_{i=1}^{N_k}\frac{\abs{\mu_0-\sigma_j}+R_{\mu_0,K}}{\abs{\mu_0-\sigma_i}}\text{,}
\end{equation}
with
\begin{equation*}
C_{l,k,j,\mu_0}=\frac{2\abs{\mu_0-\sigma_l}}{\abs{\mu_0-\sigma_l}-R_{\mu_0,K}}\left(\frac{\abs{\mu_0-\sigma_j}+R_{\mu_0,K}}{\abs{\mu_0-\sigma_l}-R_{\mu_0,K}}\right)^{S_k/N_k-1}\text{.}
\end{equation*}
Going back to the supremum, it is easy to see that all the terms depending on $l$ in the right hand side of \cref{eq:poles_notN_bound_sup} are maximized (over $l>N_k$) by the choice $l=N_k+1$, thanks to \cref{eq:pole_ordering_mu0}. This, after a suitable rearrangement of the terms, leads to
\begin{equation}\label{eq:poles_notN_bound_final}
\varphi_j\left(\min_{i=1,\ldots,N_k}\abs{\sigma_j-\lambda_i^{(k)}}\right)\leq \left.C_j'\right.^{1/N_k}C_{N_k+1,k,j,\mu_0}\prod_{i=1}^{N_k}\left(\frac{\abs{\mu_0-\sigma_j}+R_{\mu_0,K}}{\abs{\mu_0-\sigma_i}}\right)^{1/N_k}\text{.}
\end{equation}

All of the factors appearing in \cref{eq:poles_notN_bound_final}, except for the last one, can be easily shown to be bounded as $k$ (and consequently $N_k$) goes to $\infty$. Hence, it suffices to show that
\begin{equation*}
\lim_{k\to\infty}\prod_{i=1}^{N_k}\left(\frac{\abs{\mu_0-\sigma_j}+R_{\mu_0,K}}{\abs{\mu_0-\sigma_i}}\right)^{1/N_k}=0\text{.}
\end{equation*}
To this aim, the Stolz-Ces\`aro theorem \cite{Ash2012} can be applied to prove that
\begin{equation*}
\lim_{N\to\infty}\frac{1}{N}\sum_{i=1}^{N}\log\left(\frac{\abs{\mu_0-\sigma_j}+R_{\mu_0,K}}{\abs{\mu_0-\sigma_i}}\right)=\lim_{N\to\infty}\log\left(\frac{\abs{\mu_0-\sigma_j}+R_{\mu_0,K}}{\abs{\mu_0-\sigma_{N+1}}}\right)=-\infty\text{,}
\end{equation*}
yielding the claim.
\end{proof}

\davide{We have thus shown that, by minimal rational interpolation, we can approximate an arbitrary number of poles of $u$, as long as we allow the denominator degree $N$ to become sufficiently large. However,} this more general result does not provide any information on the rate of convergence, at least not a meaningful one for typical applications. Indeed, most of the steps in the proof above introduce gross simplifications, which make it impossible for the intermediate bounds \cref{eq:poles_notN_bound} and \cref{eq:poles_notN_bound_final} to be reasonably sharp.

\subsection{Convergence of minimal rational interpolants}
As a complement to the previous results, in this section we investigate the approximation properties of the whole minimal rational interpolant $\pade$, the reference being the map $u$. Similarly to the analysis of the denominator $\den$, we start by fixing the denominator degree $N$, and by proving convergence with respect to the number of samples $S$. 
\begin{theorem}\label{th:error_N}
Fix $N$ and $\varepsilon>0$, and let $E$ be an arbitrary compact subset of $\{\mu\in\C,\Phi_K(\mu)<\Phi_K(\lambda_{N+1})\}\setminus\Lambda$. There exists $S_{\varepsilon,N,E}$ such that
\begin{equation*}
\normV{u(\mu)-\pade(\mu)}\leq C_E(1+\varepsilon)^{2S}\left(\frac{\Phi_K(\mu)}{\Phi_K(\lambda_{N+1})}\right)^S\quad\text{for all }\mu\in E,\,S>S_{\varepsilon,N,E}\text{,}
\end{equation*}
with $C_E$ independent of $S$ and $\varepsilon$.
\end{theorem}

\begin{proof}
We start by deriving a useful identity for the error $u-\pade$: since $\den$ is interpolated exactly by $I^{\Xi_S}$, we can exploit identity \cref{eq:interpolant_explicit} to obtain
\begin{align}
\den(\mu)\left(u(\mu)-\pade(\mu)\right)=&u(\mu)I^{\Xi_S}(\den)(\mu)-I^{\Xi_S}(u\den)(\mu)\label{eq:error_equivalent}\\
=&\sum_{\mu'\in\Xi_S}\sum_{\lambda\in\Lambda}\frac{\den(\mu')\omega^{\Xi_S}(\mu)}{(\mu-\mu')\left.\omega^{\Xi_S}\right.'(\mu')}\left(\frac{v_\lambda}{\lambda-\mu}-\frac{v_\lambda}{\lambda-\mu'}\right)\nonumber\\
=&\davide{\omega^{\Xi_S}(\mu)\sum_{\lambda\in\Lambda}\frac{v_\lambda}{\left(\lambda-\mu\right)\omega^{\Xi_S}(\lambda)}\underbrace{\sum_{\mu'\in\Xi_S}\frac{\den(\mu')\omega^{\Xi_S}(\lambda)}{(\lambda-\mu')\left.\omega^{\Xi_S}\right.'(\mu')}}_{=I^{\Xi_S}(\den)(\lambda)=\den(\lambda)}}
\label{eq:error_equivalent_sum}\text{.}
\end{align}
After dividing by $\den(\mu)$, taking the norm, and applying \cref{lem:optimal_j}, we obtain
\begin{align}
\normV{u(\mu)-\pade(\mu)}=&\abs{\frac{\omega^{\Xi_S}(\mu)}{\den(\mu)}}\left(\sum_{\lambda\in\Lambda}\frac{\normV{v_\lambda}^2}{\abs{\lambda-\mu}^2\abs{\omega^{\Xi_S}(\lambda)}^2}\abs{\den(\lambda)}^2\right)^{1/2}\nonumber\\
\leq&\abs{\frac{\omega^{\Xi_S}(\mu)}{\den(\mu)}}\frac{1}{\min_{\Lambda}\abs{\,\cdot-\mu}}\,\jfun(\den)\label{eq:error_norm_partial}\\
\leq&\abs{\frac{\omega^{\Xi_S}(\mu)}{\den(\mu)}}\frac{1}{\min_{\Lambda}\abs{\,\cdot-\mu}}(1+\varepsilon)^S\frac{C'}{\Phi_K(\lambda_{N+1})^S}\text{,}\nonumber
\end{align}
for $S$ large enough. In particular, the constant $C'$ is independent of $S$ and $\varepsilon$.

Now, due to \cref{eq:nodal_conv}, the term $\abs{\omega^{\Xi_S}}$ is bounded by $(1+\varepsilon)^S\Phi_K^S$ for $S$ large enough. Thus, it just remains to show that $\abso{\den(\mu)}\min_{\Lambda}\abs{\,\cdot-\mu}$ is bounded away from 0 for $\mu\in E$. Since $E$ is a compact subset of $\C\setminus\Lambda$, the second factor is not troublesome. The first term is more problematic, since $\den$ may have roots inside $E$. However, \cref{th:poles_N} and the triangular inequality ensure that, for $S$ large enough, each root of $\den$ has distance from $E$ bounded away from 0. This yields the claim.
\end{proof}

If we restrict our interest to compact subsets $E$ of $K\setminus\Lambda$, we can obtain the simplified result
\begin{equation}
\max_{\mu\in E}\normV{u(\mu)-\pade(\mu)}\leq C_E\left(\frac{(1+\varepsilon)^2\capac(K)}{\Phi_K(\lambda_{N+1})}\right)^S\text{.}
\end{equation}
Thus, the rate of convergence of the approximation error is the same one that we could expect from polynomial approximants for a function holomorphic over $\{\mu\in\C,\Phi_K(\mu)<\Phi_K(\lambda_{N+1})\}$, having a pole or singularity at $\lambda_{N+1}$.

\subsubsection{Error convergence with respect to denominator degree}
Just like \cref{sec:Pole N convergence} describes an extension for variable $N$ of the convergence theory for approximate poles, here we consider a generalization of the error convergence theory in the case of diverging denominator degree. The result we are able to show is of a flavor similar to \cref{th:poles_notN}, and is summarized in the following.
\begin{theorem}\label{th:error_notN}
Let $(N_k)_{k\in\N}$, $(S_k)_{k\in\N}$, and $\Xi^{(k)}$ be as in \cref{th:poles_notN}. Fix $E\subset\C$ compact, and let \cref{as:with_scale} be satisfied for some fixed $\mu_0\in\C$. Then, for all $\varepsilon>0$, we have
\begin{equation*}
\lim_{k\to\infty}\capac\left(\left\{\mu\in E,\,\normV{u(\mu)-\padek(\mu)}\geq\varepsilon\right\}\right)=0\text{,}
\end{equation*}
with $\capac$ denoting the logarithmic capacity, see \cref{sec:Convergence preliminaries}.
\end{theorem}

\begin{proof}
We rely on some concepts introduced in the proofs of \cref{th:poles_notN} and \cref{th:error_N}, namely the pole ordering \cref{eq:pole_ordering_mu0} and the intermediate bound \cref{eq:error_norm_partial}. To start, we follow the same steps as in the proof of \cref{th:poles_notN}, exploiting \cref{eq:error_norm_partial} and \cref{eq:jfun_optimum_weird} with $\Lambda_0=\{\sigma_j\}_{j=1}^{N_k}$. The resulting bound is
\begin{equation}\label{eq:error_bound_partial_notN}
\normV{u(\mu)-\padek(\mu)}\leq\frac{C_{k,\mu_0,R}\left(\sum_{\lambda\in\Lambda}\normV{v_\lambda}^2\right)^{1/2}}{c_{\mu_0}\abs{\denk(\mu)}\min_{\Lambda}\abs{\,\cdot-\mu}}\prod_{i=1}^{N_k}\frac{R+R_{\mu_0,K}}{\abs{\mu_0-\sigma_i}}\text{,}
\end{equation}
with $R=\max_E\abs{\,\cdot-\mu_0}$ and
\begin{equation*}
C_{k,\mu_0,R}=\left(\frac{R+R_{\mu_0,K}}{\abs{\mu_0-\sigma_{N_k+1}}-R_{\mu_0,K}}\right)^{S_k-N_k}\left(\frac{2\abs{\mu_0-\sigma_{N_k+1}}}{\abs{\mu_0-\sigma_{N_k+1}}-R_{\mu_0,K}}\right)^{N_k}\text{.}
\end{equation*}

In order to obtain the desired result, we need to manage carefully the two troublesome $\mu$-dependent terms in the denominator of \cref{eq:error_bound_partial_notN}. To this aim, let $\Lambda^{(k)}=\smash{\{\lambda_j^{(k)}\}_{j=1}^{N_k}}$ be the set of roots of $\denk$. We partition $\Lambda^{(k)}$ into two (potentially empty) sets $\Lambda_{in}^{(k)}$ and $\Lambda_{out}^{(k)}$ according to the characterization
\begin{equation*}
\lambda_j^{(k)}\in\Lambda_{in}^{(k)}\quad\text{if and only if}\quad\abs{\lambda_j^{(k)}-\mu_0}\leq 2R\text{.}
\end{equation*}
Given $\#\Lambda_{in}^{(k)}$ and $\overline{N}$ the cardinalities of $\Lambda_{in}^{(k)}$ and $\Lambda\cap E$ respectively, we define the family of lemniscates
\begin{equation}\label{eq:lemniscate}
\mathcal{E}_{k,\delta}=\left\{\mu\in\C,\,\left(\prod_{\lambda'\in\Lambda_{in}^{(k)}}\abs{\mu-\lambda'}\right)\left(\prod_{\lambda\in\Lambda\cap E}\abs{\mu-\lambda}\right)\leq\delta^{\#\Lambda_{in}^{(k)}+\overline{N}}\right\}\text{,}
\end{equation}
which depends on the index $k$ and on the (small) value $\delta>0$.

We remark that the exponent of $\delta$ in \cref{eq:lemniscate} is equal to the degree of the monic polynomial (in $\mu$) whose magnitude is the left hand side of the inequality. Thus \cite[Theorem 6.6.3]{Baker1996}, the logarithmic capacity of $\mathcal{E}_{k,\delta}$ is equal to $\delta$. (We are assuming without loss of generality that $\smash{\#\Lambda_{in}^{(k)}+\overline{N}}>0$. If this is not the case, $\mathcal{E}_{k,\delta}$ is empty for $\delta<1$, and the claim holds quite trivially.)

The main structure of the remainder of the proof is the following:
\begin{enumerate}[(i)]
\item\label{en:delta} we define explicitly a sequence $(\delta_k)_{k=1,2,\ldots}$ converging to 0;
\item\label{en:epsilon} we show that $\normV{u(\mu)-\padek(\mu)}<\varepsilon$ for all $\mu\in E\setminus\mathcal{E}_{k,\delta_k}$, for large $k$.
\end{enumerate}
Then the claim follows, since, by inclusion,
\begin{equation*}
0\leq\lim_{k\to\infty}\capac\left(\left\{\mu\in E,\,\normV{u(\mu)-\padek(\mu)}\geq\varepsilon\right\}\right)\leq\lim_{k\to\infty}\capac\left(\mathcal{E}_{k,\delta_k}\right)=\lim_{k\to\infty}\delta_k=0\text{.}
\end{equation*}

However, before we can proceed with either step, it is necessary to introduce some additional bounds in \cref{eq:error_bound_partial_notN}. First, since $\denk$ is normalized, by \cref{as:with_scale} it must hold
\begin{equation*}
\abs{\denk(\mu)}\geq\frac{1}{C_{\mu_0}}\left(\prod_{\lambda'\in\Lambda_{in}^{(k)}}\frac{\abs{\mu-\lambda'}}{R_{\mu_0}+\abs{\mu_0-\lambda'}}\right)\left(\prod_{\lambda''\in\Lambda_{out}^{(k)}}\frac{\abs{\mu-\lambda''}}{R_{\mu_0}+\abs{\mu_0-\lambda''}}\right)\text{.}
\end{equation*}
In the first group of factors, $\abs{\mu_0-\lambda'}$ is bounded from above by $2R$. In the second group, since $\abs{\mu-\mu_0}\leq R$ for all $\mu\in E$ by definition, we have
\begin{equation*}
\frac{\abs{\mu-\lambda''}}{R_{\mu_0}+\abs{\mu_0-\lambda''}}\geq\frac{\abs{\mu_0-\lambda''}}{R_{\mu_0}+\abs{\mu_0-\lambda''}}-\frac{\abs{\mu-\mu_0}}{R_{\mu_0}+\abs{\mu_0-\lambda''}}\geq\frac{R}{R_{\mu_0}+2R}\text{.}
\end{equation*}
Hence,
\begin{equation*}
\abs{\denk(\mu)}\geq\frac{1}{C_{\mu_0}}\frac{R^{N_k-\#\Lambda_{in}^{(k)}}}{\left(R_{\mu_0}+2R\right)^{N_k}}\prod_{\lambda'\in\Lambda_{in}^{(k)}}\abs{\mu-\lambda'}\text{,}
\end{equation*}
which, provided $\delta\leq R$, for all $\mu\in E\setminus\mathcal{E}_{k,\delta}$, implies
\begin{align*}
\abs{\denk(\mu)}\prod_{\lambda\in\Lambda\cap E}\abs{\mu-\lambda}\geq&\frac{R^{N_k-\#\Lambda_{in}^{(k)}}}{C_{\mu_0}\left(R_{\mu_0}+2R\right)^{N_k}}\left(\prod_{\lambda'\in\Lambda_{in}^{(k)}}\abs{\mu-\lambda'}\right)\left(\prod_{\lambda\in\Lambda\cap E}\abs{\mu-\lambda}\right)\\
>&\frac{R^{N_k-\#\Lambda_{in}^{(k)}}\delta^{\#\Lambda_{in}^{(k)}+\overline{N}}}{C_{\mu_0}\left(R_{\mu_0}+2R\right)^{N_k}}\geq\frac{\delta^{N_k+\overline{N}}}{C_{\mu_0}\left(R_{\mu_0}+2R\right)^{N_k}}\text{.}
\end{align*}

Accordingly, we can simplify \cref{eq:error_bound_partial_notN} for all $\mu\in E\setminus\mathcal{E}_{k,\delta}$:
\begin{multline*}
\normV{u(\mu)-\padek(\mu)}<\frac{C_{\mu_0}\left(\sum_{\lambda\in\Lambda}\normV{v_\lambda}^2\right)^{1/2}\prod_{\lambda\in\Lambda\cap E}\abs{\mu-\lambda}}{c_{\mu_0}\min_{\Lambda}\abs{\,\cdot-\mu}}\times\\
\times C_{k,\mu_0,R}\frac{\left(R_{\mu_0}+2R\right)^{N_k}}{\delta^{N_k+\overline{N}}}\prod_{i=1}^{N_k}\frac{R+R_{\mu_0,K}}{\abs{\mu_0-\sigma_i}}\text{.}
\end{multline*}
Since $E$ is compact and $\Lambda$ has no finite limit point, the distance between $E$ and $\Lambda\setminus E$ is strictly positive. Hence, the quantity
\begin{equation*}
C_{\mu_0}\left(\sum_{\lambda\in\Lambda}\normV{v_\lambda}^2\right)^{1/2}\prod_{\lambda\in\Lambda\cap E}\abs{\mu-\lambda}\,\Big/\,\left(c_{\mu_0}\min_{\Lambda}\abs{\,\cdot-\mu}\right)
\end{equation*}
is bounded, for all $\mu\in E$, by a constant $C'$ depending only on $E$, $\mu_0$, and $\Lambda$.

Now it is trivial to achieve \cref{en:epsilon} by setting
\begin{equation}\label{eq:error_bound_delta}
\delta=\delta_k=\bigg(\frac{C'C_{k,\mu_0,R}\left(R_{\mu_0}+2R\right)^{N_k}}{\varepsilon}\prod_{i=1}^{N_k}\frac{R+R_{\mu_0,K}}{\abs{\mu_0-\sigma_i}}\bigg)^{1/(N_k+\overline{N})}\text{.}
\end{equation}
(We remark that we still need to guarantee that the working assumption $\delta\leq R$ is satisfied. However, since we are planning to prove \cref{en:delta}, such condition will trivially hold, provided $k$ is large enough.)

At this point, it suffices to apply the same strategy employed in the final part of the proof of \cref{th:poles_notN}. In particular, we have
\begin{equation*}
\lim_{k\to\infty}\left(\prod_{i=1}^{N_k}\frac{R+R_{\mu_0,K}}{\abs{\mu_0-\sigma_i}}\right)^{1/(N_k+\overline{N})}=\lim_{N\to\infty}\frac{R+R_{\mu_0,K}}{\abs{\mu_0-\sigma_{N+1}}}=0\text{,}
\end{equation*}
while the other factors in \cref{eq:error_bound_delta} stay bounded as $k$ increases.
\end{proof}

Due to \cref{eq:capacity_bounds}, \cref{th:error_notN} can be weakened to prove the convergence in probability of $\smash{\padek}$ to $u$, when both are interpreted as functions from the probability space $(E,\mcB(E),\ell/\ell(E))$ to the Banach space $(V,\normV{\cdot})$.

Still, convergence in capacity is somewhat stronger than that in probability: for instance, it ensures that the set of parameter values for which the approximant yields an error above $\varepsilon$ in the limit, namely
\begin{equation*}
\bigcap_{l\in\N}\bigcup_{k\geq l}\left\{\mu\in E\,:\,\normV{u(\mu)-\padek(\mu)}\geq\varepsilon\right\}\text{,}
\end{equation*}
cannot include any curve in $\C$.

\subsection{Generalizations to non-orthogonal residues}\label{sec:nonortho}
\davide{In the previous sections, we have been working under the assumption that the residues $\{v_\lambda\}_{\lambda\in\Lambda}$ form a $V$-orthogonal set. However, most of our results for fixed $N$ can be extended to the general case, in which we impose no condition on the angles between residues. As such, in this section we remove the assumption of residue orthogonality.

We start by extending \cref{lem:optimal_j}.
\begin{lemma}\label{lem:optimal_j_no}
Let $Q$, $\Lambda_0$, and $\varepsilon$ be the same as in \cref{lem:optimal_j}. Then
\begin{equation*}
\jfun(Q)^2=\sum_{\lambda,\lambda'\in\Lambda}\frac{\dualV{v_\lambda}{v_{\lambda'}}}{\omega^{\Xi_S}(\lambda)\overline{\omega^{\Xi_S}(\lambda')}}Q(\lambda)\overline{Q(\lambda')}\text{.}
\end{equation*}
Moreover,
\begin{equation*}
\jfun(\den)\leq\frac{\sum_{\lambda\in\Lambda\setminus\Lambda_0}\normV{v_\lambda}}{\normT{\omega^{\Lambda_0}}_N}\sup_{\Lambda\setminus\Lambda_0}\abs{\frac{\omega^{\Lambda_0}}{\omega^{\Xi_S}}}\text{,}
\end{equation*}
and there exists $S_{\varepsilon,N}$ such that
\begin{equation*}
\jfun(\den)\leq (1+\varepsilon)^S\frac{\sum_{j>N}\normV{v_{\lambda_j}}}{\normT{\prod_{j=1}^N\left(\ \cdot-\lambda_j\right)}_N}\frac{\prod_{j=1}^N\abs{\lambda_{N+1}-\lambda_j}}{\Phi_K(\lambda_{N+1})^S}\quad\text{for }S>S_{\varepsilon,N}\text{,}
\end{equation*}
provided $\Phi_K(\lambda_{N+1})<\Phi_K(\lambda_{N+2})$.
\end{lemma}

\begin{proofsketch}
One can follow most of the proof of \cref{lem:optimal_j}. However, \cref{eq:optimal_j_bound_middle} does not hold due to the lack of orthogonality. Still, we can replace it by
\begin{equation*}
\jfun(\den)^2\leq\jfun\left(\frac{\omega^{\Lambda_0}}{\normT{\omega^{\Lambda_0}}_N}\right)^2\leq\frac{1}{\normT{\omega^{\Lambda_0}}_N^2}\left(\sum_{\lambda\in\Lambda\setminus\Lambda_0}\normV{v_\lambda}\abs{\frac{\omega^{\Lambda_0}(\lambda)}{\omega^{\Xi_S}(\lambda)}}\right)^2\text{,}
\end{equation*}
which holds by triangular inequality.
\end{proofsketch}

We remark the the discrete 2-norm of the residue $V$-norms in \cref{lem:optimal_j} has been replaced by the stronger 1-norm. In order to derive convergence properties, we need to introduce the assumption that the $V$-norms of the residues are summable, so that the inequalities in \cref{lem:optimal_j_no} are nontrivial.
\begin{theorem}\label{th:poles_N_no}
Let $N$, $\{\lambda_j^{\Xi_S}\}_{j=1}^N$, and $\varepsilon$ be the same as in \cref{th:poles_N}. If $\sum_{\lambda\in\Lambda}\normV{v_\lambda}<\infty$, there exists $S_{\varepsilon,N}$ such that
\begin{equation*}
\abs{\den\left(\lambda_j\right)}\leq C_j(1+\varepsilon)^{4S}\left(\frac{\Phi_K(\lambda_j)}{\Phi_K(\lambda_{N+1})}\right)^S\quad\text{for }S>S_{\varepsilon,N}
\end{equation*}
and
\begin{equation*}
\min_{i=1,\ldots,N}\abs{\lambda_i^{\Xi_S}-\lambda_j}\leq C_j(1+\varepsilon)^{4S}\left(\frac{\Phi_K(\lambda_j)}{\Phi_K(\lambda_{N+1})}\right)^S\quad\text{for }S>S_{\varepsilon,N}\text{,}
\end{equation*}
for $j=1,\ldots,N$, with $C_j$ independent of $S$ and $\varepsilon$.
\end{theorem}

\begin{proofsketch}
The proof is similar to that of \cref{lem:den_N_lambda} and \cref{th:poles_N}. However, we can observe that the exponents in the convergence rates are halved. This is caused by ``interference'' between the modes: the difference between the representative matrices $\Gfull$ and $\Gpart$ is
\begin{equation*}
\Gfull-\Gpart=\sum_{\substack{\lambda,\lambda'\in\Lambda \\ (\lambda,\lambda')\notin\{1,\ldots,N\}^2}}\frac{\dualV{v_\lambda}{v_{\lambda'}}}{\omega^{\Xi_S}(\lambda)\overline{\omega^{\Xi_S}(\lambda')}}\oomega_{\lambda'}^{}\oomega_\lambda^H\text{,}
\end{equation*}
which, in particular, contains ``off-diagonal'' terms, for instance the one corresponding to $\lambda=1$ and $\lambda'=N+1$. When computing an upperbound for a term of this form, we cannot hope to obtain more than a \emph{single} factor of $\Phi_K(\lambda_{N+1})^{-S}$, whereas two such factors can be gained in the orthogonal case. Similarly, the correlation between the first $N$ resonances causes one less factor of $\abs{\omega^{\Xi_S}(\lambda_j)}$ to be obtainable from $\Gpart$.
\end{proofsketch}

\begin{theorem}\label{th:error_N_no}
Let $N$, $\varepsilon$, and $E$ be as in \cref{th:error_N}. If $\sum_{\lambda\in\Lambda}\normV{v_\lambda}<\infty$, there exists $S_{\varepsilon,N,E}$ such that
\begin{equation*}
\normV{u(\mu)-\pade(\mu)}\leq C_E(1+\varepsilon)^{4S}\left(\frac{\Phi_K(\mu)}{\Phi_K(\lambda_{N+1})}\right)^S\quad\text{for all }\mu\in E,\,S>S_{\varepsilon,N,E}\text{,}
\end{equation*}
with $C_E$ independent of $S$ and $\varepsilon$.
\end{theorem}

\begin{proofsketch}
Since \cref{lem:optimal_j} does not hold, it is impossible to apply the middle part of the proof of \cref{th:error_N}. Instead, we can split the sum in \cref{eq:error_equivalent_sum} into two parts: for the first $N$ poles $\lambda_1,\ldots,\lambda_N$, \cref{th:poles_N_no} can be applied to bound $|\den/\omega^{\Xi_S}|$ from above; for the remaining resonances, it suffices to use the argument from the last portion of \cref{lem:optimal_j}, with some modifications to account for the presence of $\den$.
\end{proofsketch}

Unfortunately, convergence results for variable $N$ do not appear to be possible in the non-orthogonal case, without some strong hypotheses on the ``amount of correlation'' between resonances: for instance, we believe that one could work under the assumption that the worst-case angle between residues
\begin{equation*}
\alpha(\lambda)=\inf_{\lambda'\in\Lambda\setminus\{\lambda\}}\frac{\dualV{v_\lambda}{v_{\lambda'}}}{\normV{v_\lambda}\normV{v_{\lambda'}}}
\end{equation*}
is bounded away from 0 for all poles $\lambda$ in the region of interest.
}

\section{A posteriori error indicators}\label{sec:Greedy}
All convergence results presented above are \emph{a priori} estimates. In particular, they contain quantities which are often not available, namely the number and locations of the poles inside (and of some of the ones outside) the parameter domain $K$. Moreover, the results in \cref{th:poles_N,th:error_notN} hold only (with some extensions) if the solution map is of the form \cref{eq:map_decomposition}, whereas minimal rational interpolants can, in principle, be applied to more general parametric problems and sample points sets.

A typical MOR problem can be cast in the following form: let a certain parameter set $K$ and a parametric problem as in \cref{eq:parametric_problem} be given, along with some fixed tolerance $\varepsilon>0$. The task is to obtain an approximate solution map $\widetilde{u}$ such that
\begin{equation*}
\max_{\mu\in K}\normV{u(\mu)-\widetilde{u}(\mu)}\leq\varepsilon\quad\text{or}\quad\max_{\mu\in K}\frac{\normV{u(\mu)-\widetilde{u}(\mu)}}{\normV{u(\mu)}}\leq\varepsilon\text{.}
\end{equation*}
In our framework, the presence of singularities in the true and, possibly, the approximate solution map within $K$ may make the conditions above meaningless. In such cases, it may be preferable \cite{Chevreuil2012} to consider a residual-based accuracy requirement
\begin{equation}\label{eq:residual_tolerance}
\max_{\mu\in K}\normW{\mcF_\mu\left(\widetilde{u}(\mu)\right)}\leq\varepsilon\quad\text{or}\quad\max_{\mu\in K}\frac{\normW{\mcF_\mu\left(\widetilde{u}(\mu)\right)}}{\normW{\mcF_\mu\left(0\right)}}\leq\varepsilon\text{,}
\end{equation}
with the $\normW{\cdot}$-norm replacing $\normV{\cdot}$ to account for regularity differences between residual and solution.

Now, let us take as approximate map $\widetilde{u}$ the minimal rational interpolant $\padet$ with a certain denominator degree $N$ and samples at the $S$ points $Z_S$, which need not be Fej\'er points for $K$. We face the task of understanding whether the required tolerance $\varepsilon$ is achieved, in the sense of \cref{eq:residual_tolerance}.

In particular, let us consider the problem of computing \emph{a posteriori} the norm of the residual $\normW{\mcF_\mu\left(\widetilde{u}(\mu)\right)}$ at some given point $\mu\in K$. In order to proceed, we will assume that the operator $\mcF_\mu$ is linear and has a separable form in $\mu$, i.e. that there exist two families of complex-valued functions $\{\theta_i^F\}_{i=1}^{n_F}$ and $\{\theta_i^f\}_{i=1}^{n_f}$, a family of operators $\{F_i\}_{i=1}^{n_F}$ over (suitable subsets of) $V$, and a family $\{f_i\}_{i=1}^{n_f}$ of elements of $V$, such that
\begin{equation}\label{eq:affine_decomposition}
\mcF_\mu(v)=F(\mu)v-f(\mu)=\sum_{i=1}^{n_F}\theta_i^F(\mu)F_iv-\sum_{i=1}^{n_f}\theta_i^f(\mu)f_i\text{.}
\end{equation}
This is the ideal situation to employ RB methods \cite{Chen2010, Hesthaven2016, Quarteroni2015, Rozza2008}, and many strategies have been devised to approximate general parametric problems into the form \cref{eq:affine_decomposition}, e.g. the (D)EIM \cite{Chaturantabut2010,Grepl2007} and hyper-reduction techniques \cite{Carlberg2013}.

Now, $F(\mu)$ can be applied to both sides of \cref{eq:error_equivalent} to obtain
\begin{multline}\label{eq:affine_residual_decomposition}
\dent(\mu)\mcF_\mu\left(\padet(\mu)\right)=\sum_{\mu'\in Z_S}\frac{\dent(\mu')\,\omega^{Z_S}(\mu)}{(\mu-\mu')\left.\omega^{Z_S}\right.'(\mu')}\mcF_\mu\left(u(\mu')\right)\\
=\sum_{i=1}^{n_F}F_i\Delta^{Z_S}\left(\theta_i^FI^{Z_S}\left(u\dent\right)\right)(\mu)-\sum_{i=1}^{n_f}\Delta^{Z_S}\left(\theta_i^f\dent\right)(\mu)\,f_i\text{.}
\end{multline}
with $\Delta^{Z_S}(\phi)=\phi-I^{Z_S}(\phi)$ the interpolation error for a function $\phi$. \hide{(The properties $\Delta^{Z_S}(\dent)\equiv 0$ and $\Delta^{Z_S}\big(I^{Z_S}\big(u\dent\big)\big)\equiv 0$ have both been exploited in the derivation above.)}


This, after dividing by $\smash{\dent}(\mu)$, provides an affine decomposition of the residual with $n_F+n_f$ terms. Hence, its norm (squared) admits an affine decomposition with $(n_F+n_f)^2$ terms. In particular, assuming an offline-online framework \cite{Hesthaven2016, Quarteroni2015, Rozza2008}, the separable terms of the residual can be precomputed offline, i.e. together with the $S$ expensive evaluations of the target solution map $u$, without increasing the complexity of the overall MOR procedure.

Of course, the discussion above can intrinsically be applied only in an intrusive fashion, since we require access to the operators defining the parametric problem. In general, obtaining an \emph{a posteriori} non-intrusive error indicator can be quite tricky. Here we propose a simplified approach, which is rigorous only for very specific problem structures, namely
\begin{equation}\label{eq:affine_decomposition_linear}
\mcF_\mu(u)=F(\mu)u-f=(F_0+\mu F_1)u-f\text{.}
\end{equation}
For such problems, thanks to \cref{eq:interpolant_explicit} and \cref{eq:affine_residual_decomposition}, it must hold
\begin{align*}
\dent(\mu)\mcF_\mu\left(\padet(\mu)\right)=&\omega^{Z_S}(\mu)F_1\sum_{\mu'\in Z_S}\frac{\dent(\mu')u(\mu')}{\left.\omega^{Z_S}\right.'(\mu')}\\
=&\frac{\omega^{Z_S}(\mu)}{(S-1)!}F_1\dmu{S-1}I^{Z_S}(u\dent)\text{.}
\end{align*}
Now it suffices to take the $\normW{\cdot}$-norm and exploit \cref{eq:target_functional} to obtain
\begin{align}
\normW{\mcF_\mu\left(\padet(\mu)\right)}=&\frac{1}{(S-1)!}\normW{F_1\dmu{S-1}I^{Z_S}(u\dent)}\abs{\frac{\omega^{Z_S}(\mu)}{\dent(\mu)}}\label{eq:residual_bound_linear}\\
\leq&\left(\sup_{v\in V\setminus\{0\}}\frac{\normW{F_1v}}{\normV{v}}\right)\jfunt\left(\dent\right)\abs{\frac{\omega^{Z_S}(\mu)}{\dent(\mu)}}\text{.}\label{eq:residual_bound_linear_ineq}
\end{align}

%

One obvious limitation of \cref{eq:residual_bound_linear} and \cref{eq:residual_bound_linear_ineq} is the need to evaluate or bound expensive norms (either the $W$ one or the $V\to W$ operator one) involving $F_1$. We remark that this drawback can be interpreted as an issue in identifying the exact scaling of the error estimator, a common problem even for stable parametric problems, where the inf-sup (or coercivity) constant must be estimated to link residual and error \cite{Huynh2007, Quarteroni2015, Rozza2008}. If one can overcome such limitation, see \cref{sec:scattering}, \cref{eq:residual_bound_linear} and \cref{eq:residual_bound_linear_ineq} are certainly quite appealing from a computational standpoint, since their evaluation consists only of few scalar computations.

If \cref{eq:affine_decomposition_linear} does not hold, we can replace $F(\mu)$ by a suitable linear approximation (e.g. its first-order Taylor series around some parameter) in a spirit similar to that of the Empirical Interpolation Method. Then \cref{eq:residual_bound_linear} and \cref{eq:residual_bound_linear_ineq} can be applied heuristically, with an accuracy which may depend quite sharply on the smoothness of $\mcF_\mu$, in particular on its second-order variations over the parameter domain.

As a supplement to our discussion, we would like to note that \emph{a posteriori} error indicators are often used in RB procedures not only to determine whether the required accuracy has been achieved, but also to drive the selection of the sample points. For instance, in the weak-greedy RB approach \cite{Quarteroni2015, Rozza2008}, one keeps adding a new snapshot at the parameter value which maximizes the \emph{a posteriori} error indicator of choice, thus updating the surrogate model (and the error indicator as well), until convergence. We envision that a similar adaptive procedure could be devised also for minimal rational interpolation, with new parameter values being added to the sample set according to the same logic.

\section{Extensions to multiple parameters}\label{sec:conclusions_multi}
\davide{Throughout our discussion, we have only dealt with the approximation of univariate meromorphic solution maps. However, it is often of interest to approximate the solution map of multi-parameter problems: for example, in the analysis of parametric dynamical systems, one may need to approximate the solution of
\begin{equation}\label{eq:multi_par}
(A(\pp)-\mu E(\pp))u(\mu,\pp)=B(\pp),\quad\text{with }\pp\in\C^{d-1},
\end{equation}
where $\pp$ is a collection of physical or geometric parameters. For such problems, it is often quite complicated to prove properties of $u$, jointly in $\mu$ and $\pp$; however, the analysis is much simpler if the extra parameters $\pp$ are kept fixed (``marginalized out'') and one looks at the univariate solution map $u_{\pp}(\mu):\mu\mapsto u(\mu,\pp)$.

This has lead to the development of \emph{parametric MOR} (pMOR) \cite{Baur2011, Ferranti2011, Panzer2010, Xiao2019, Yue2019}, whose core idea is the following: take a set of parameter values $\{\pp_1,\ldots,\pp_T\}\subset\C^{d-1}$ and, for each $j=1,\ldots,T$, build a surrogate $\widetilde{u}_{\pp_j}$ of $u_{\pp_j}$ (this involves computing the snapshots $u(\mu_i,\pp_j)$ for $i=1,\ldots,S_j$, and applying a single-parameter MOR strategy, with respect to $\mu$ only); then, the joint reduced model $\widetilde{u}$ of $u$ is constructed as a suitable combination of $\widetilde{u}_{\pp_1},\ldots,\widetilde{u}_{\pp_T}$.

In a pMOR framework, one could apply minimal rational approximation to obtain the surrogates in frequency only $\{\widetilde{u}_{\pp_j}\}_{j=1}^T$. In particular, the results presented in the previous sections apply: for instance, the residual estimator in \cref{sec:Greedy} is still valid, allowing for a greedy approach with respect to $\mu$. This line of approximation for multi-parameter problems is currently under investigation.

There are several reasons to prefer pMOR over global MOR (in $d$ parameters) for problems like \cref{eq:multi_par}. Many of them are related to the curse of dimensionality, which forces to take a large number of snapshots, and to build a surrogate model of considerable size, undermining the cost-effectiveness of MOR. We refer to the general pMOR literature cited above for a broader discussion.

Here, we wish to remark that these disadvantages affect also the natural extension of minimal rational interpolation to $d$ parameters, which we could define as follows: assume that the sets of polynomials $\Pspaced{N}$ and $\PspacedN{N}$, as well as the interpolation operator $I^{\Xi_S}$, are suitably defined (for instance, $\Pspaced{N}$ could denote the space of polynomials of \emph{total} degree at most $N$); the optimization problem which yields the surrogate denominator $\den$ is now: find $\den\in\PspacedN{N}$ which minimizes over $\PspacedN{N}$ the convex functional
\begin{equation*}
Q\;\mapsto\;\jfun(Q)^2=\stackrel[\alpha_1+\ldots+\alpha_d=M]{}{\sum_{\alpha_1=0,1\ldots}\cdots\sum_{\alpha_d=0,1\ldots}}\normV{\frac{1}{\alpha_1!\cdots\alpha_d!}\derivi{p}{1}{\alpha_1}\cdots\derivi{p}{d-1}{\alpha_{d-1}}\dmu{\alpha_d}I^{\Xi_S}(uQ)}^2\text{,}
\end{equation*}
where $M$ is the maximal total degree of the interpolant $I^{\Xi_S}(uQ)$, so that, in general, $S\sim M^d$. In practice, an algorithm for minimizing this functional can be found as a generalization of the strategy proposed in \cref{sec:Preliminary results}; while we omit the details here, we remark that the Gramian matrix $\Psi$ representing $\jfun$, see \cref{eq:quadratic_form} for the univariate case, becomes of size approximately $S^{2-1/d}\times N^d$.

In addition to the other limitations of multivariate minimal rational interpolation, we also observe that it is impossible to extend directly our single-parameter theory to the multivariate case, since many of the core ideas (e.g. the barycentric expansion \cref{eq:interpolant_explicit}) do not generalize nicely to more than one dimension. Overall, also considering the numerical issues intrinsic to high-dimensional interpolation, we believe minimal rational approximants to be an appropriate MOR technique only in low dimension: extensions to many parameters should rely on pMOR or similar approaches.
}

\section{Numerical examples}\label{sec:examples}
\davide{In this section, we perform some numerical experiments to show the usefulness of minimal rational interpolation, and to verify the theoretical convergence rates and error indicator. The first example satisfies the assumption under which our results have been derived; in particular, the residues are orthogonal. Instead, the solution map in the second example can be proven meromorphic, but not necessarily of the form \cref{eq:map_decomposition}. For the sake of reproducibility, the code used to run the simulations is available at \cite{Zenodo}.}

\subsection{Normal eigenvalue problem}\label{sec:eigenvalue}
Let $n=100$ and take $A\in\C^{n\times n}$ a normal matrix whose eigenvalues are randomly generated according to a uniform distribution over $\{x+\iota y, (x,y)\in[-5,5]^2\}$, and whose eigenvectors are prescribed by orthonormalizing a matrix with random (complex) standard gaussian entries. In particular, we order the eigenvalues $\{\lambda_j\}_{j=1}^n$ according to their distance from 0.

Our task is to estimate the eigenvalues of $A$ which lie within the unit disk $\smash{K=\overline{\ball{0}{1}}}$. To this aim, we fix $\vv\in\C^n$ a random gaussian vector, and we approximate by minimal rational interpolants the solution map of the parametric problem
\begin{equation}
\text{for }\mu\in K,\;\text{find }\uu(\mu)\in\C^n\;\text{s.t.}\;\left(A-\mu I\right)\uu(\mu)=\vv\text{,}
\end{equation}
with $I\in\C^{n\times n}$ the identity matrix. \davide{We endow $\C^n$ with the Euclidean scalar product.} As sample points $\Xi_S$, we take the $S$-th roots of unity, and we choose to employ the polynomial norm $\normT{\cdot}_{0,N}$, see \cref{eq:monomial_norm}. This norm is induced by monomials, which are $L^2(\partial K)$-orthogonal. Moreover, the action of the interpolation operator $\smash{I^{\Xi_S}}$ can be evaluated quite efficiently as a Fast Fourier Transform.

%
%
Let us assume that our interest lies in approximating just the eigenvalues of $A$ inside $K$. As their exact number is unknown, we can choose the denominator degree $N$ according to different strategies:
\renewcommand{\labelenumi}{\textbullet}
\begin{enumerate}
\item we fix $N$ (in our case, we set $N=10$) and increase progressively the number of samples $S$; to verify if the starting guess on $N$ was large enough, we can simply check if at least one of the approximate poles has converged (for large $S$) to some point outside $K$, see \cref{th:poles_N}.
\item we can increase $N$ along with the number of samples $S$, for instance by choosing $N=S-1$ for all $S$; in this case we are sure to approximate all poles, see \cref{th:poles_notN}.
\end{enumerate}

We compare visually the two approaches in \cref{fig:ex1_error}, where we show the error norm obtained by approximants of type $[20/10]$ and $[20/20]$, both based on $S=21$ samples at roots of unity. We can see that the error increases much quicker near the boundary of the sampled square in the case $N=10$. This is to be expected, as the region of convergence increases together with $N$, see \cref{th:error_N}. A (possibly) more interesting observation is that the error appears to be globally smaller when we choose $N=20$. Overall, we can conclude that the choice $N=S-1$ leads to a better approximant.
\begin{figure}[t]
\begin{center}
\begin{minipage}{.45\textwidth}
\begin{center}
\hspace{-.15cm}%
\includegraphics[scale = .95]{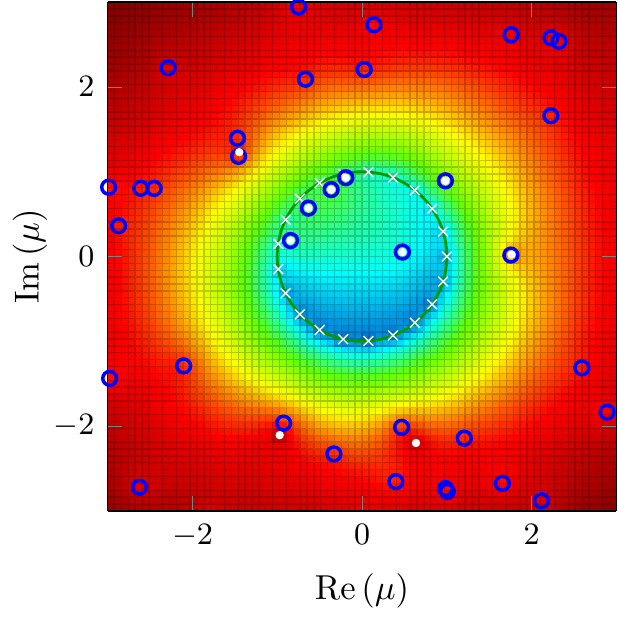}%
\end{center}
\end{minipage}%
\begin{minipage}{.55\textwidth}
\begin{center}
\vspace{-.05cm}
\includegraphics[scale = .95]{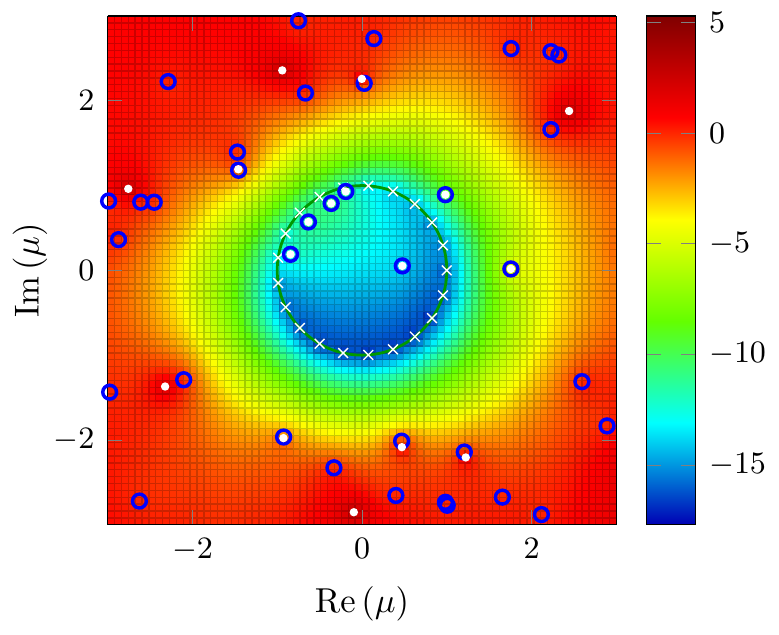}
\end{center}
\end{minipage}
\caption{Natural logarithm of the relative error norm $\|\uu-\uu_N^{\Xi_S}\|_2/\norm{\uu}{2}$ achieved by minimal rational interpolants based on samples at $\Xi_S=\{e^{2j\pi\iota/21}\}_{j=1}^{21}$ (black crosses) with $N=10$ (left) and $N=20$ (right). The exact and approximated poles are represented in blue and white respectively.}
\label{fig:ex1_error}

\vspace{.25cm}
\includegraphics[scale = .95]{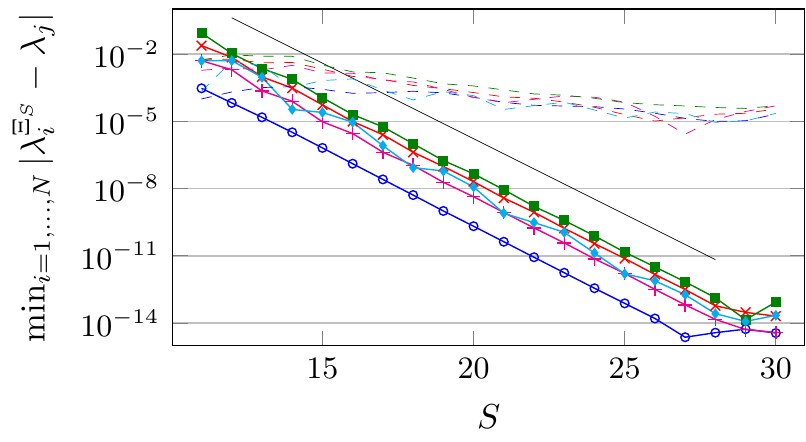}\hspace{0.cm}%
\includegraphics[scale = .95]{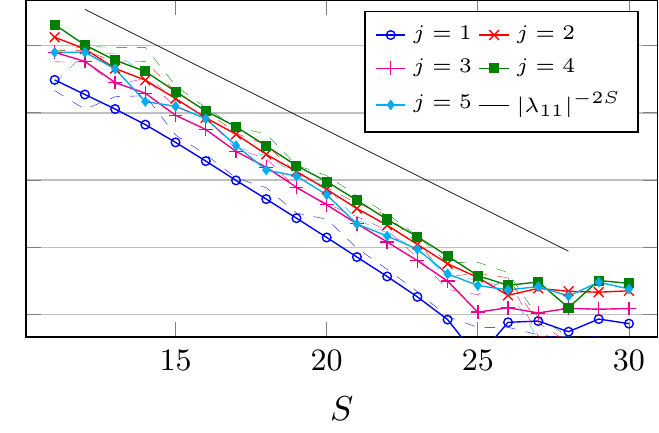}
\caption{Pole approximation error for the poles inside $K$. On the left $N=10$ is kept fixed, whereas $N=S-1$ on the right. In black the theoretical rate \eqref{eq:poles_N_inside}. The dashed lines show the results obtained by projection on subspaces of dimension $N$ computed through POD.}
\label{fig:ex1_poles_err}
\end{center}
\end{figure}

Now it just remains to check how well each approach manages to capture the location of the eigenvalues of $A$. To this aim, we increase $S$ from $11$ to $30$, and compute the pole approximation error, defined in the left hand side of \cref{eq:poles_N}, for the 5 poles within $K$. The results are shown in \cref{fig:ex1_poles_err}. For fixed $N=10$, the rate predicted by \cref{eq:poles_N_inside} can be observed (Green's potential for a centered disk coincides with the complex magnitude outside the disk). Remarkably, a similar, if not slightly better, rate of convergence is obtained for increasing $N=S-1$, despite the absence of theoretical results in this regard.

In \cref{fig:ex1_poles_err} we also show the pole approximation error obtained through a RB-like projection technique: given the same sample set $\Xi_S$, we perform a Proper Orthogonal Decomposition (POD) \cite{Quarteroni2015} of the samples, identifying the $N$ dominant ``sample directions''; then we project orthogonally the original problem on the subspace spanned by these $N$ directions, and we use the solution of the resulting $N\times N$ eigenvalue problem to approximate the original one. We can observe that the projection technique performs much worse than minimal rational interpolants for fixed $N=10$, whereas for $N=S-1$ the two methods achieve very similar results.

\begin{figure}[t]
\begin{center}
\begin{minipage}{.45\textwidth}
\begin{center}
\hspace{-.15cm}%
\includegraphics[scale = .95]{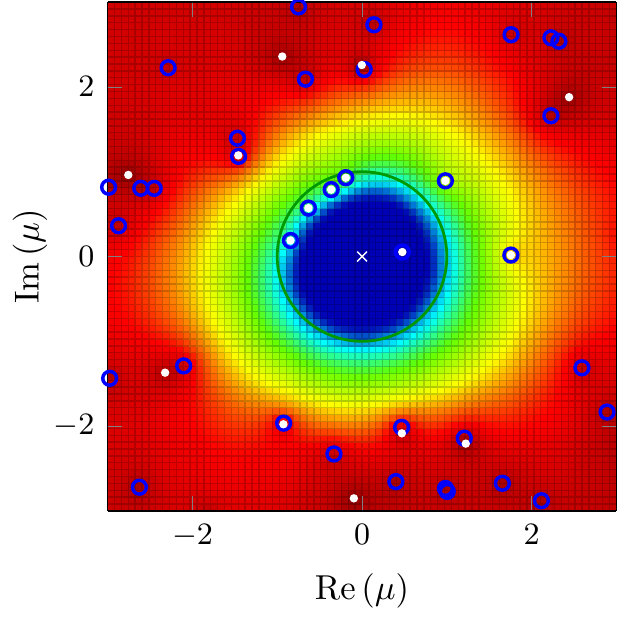}%
\end{center}
\end{minipage}%
\begin{minipage}{.55\textwidth}
\begin{center}
\vspace{-.05cm}
\includegraphics[scale = .95]{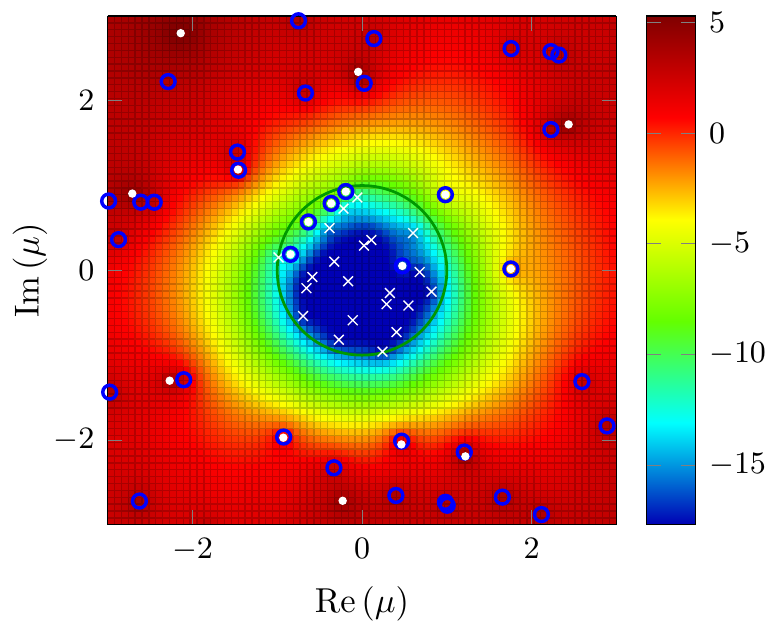}
\end{center}
\end{minipage}
\caption{\davide{Logarithm of the relative error norm $\|\uu-\uu_N^{\Xi_S}\|_2/\norm{\uu}{2}$ achieved by minimal rational interpolants based on samples (white crosses) at the origin (left) and at quasi-random points (right) with $N=20$. The exact and approximated poles are represented in blue and white respectively.}}
\label{fig:ex1_error_oth}

\vspace{.25cm}
\includegraphics[scale = .95]{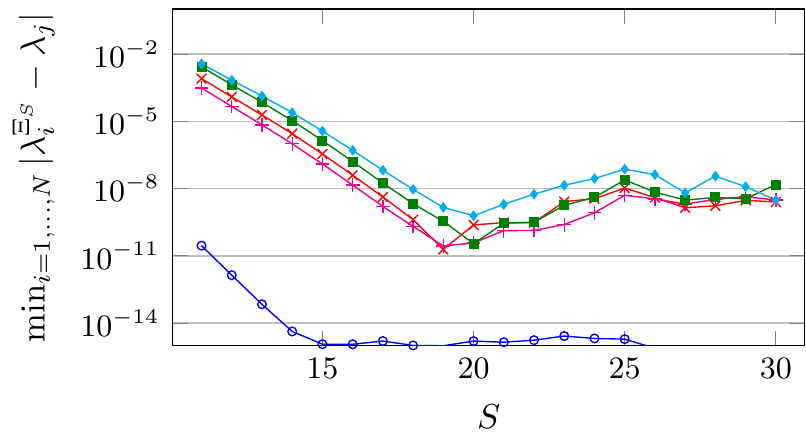}%
\includegraphics[scale = .95]{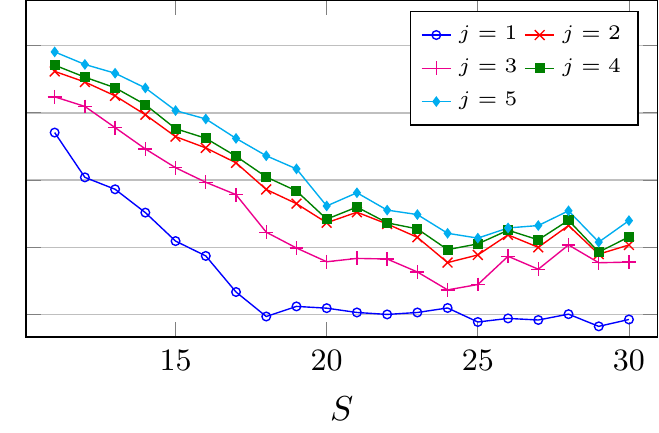}
\caption{\davide{Pole approximation error for the poles inside $K$. The samples are all at the origin on the left, whereas on the right the samples are at quasi-random points. In both cases we keep $N=S-1$.}}
\label{fig:ex1_poles_err_oth}
\end{center}
\end{figure}

\davide{
Lastly, we wish to test two other approximation strategies: a fast LS Pad\'e approximant \cite{Bonizzoni2018b} centered at the origin, and a minimal rational interpolant based on quasi-random (Halton) points in the unit disk. We show the error norms and the approximated poles in \cref{fig:ex1_error_oth} for $N=S-1=20$. Qualitatively, it appears that, for both choices, the accuracy in the approximation of the resonances is similar to that achieved with samples at the roots of unity. With respect to the $[20/20]$ minimal rational interpolant in \cref{fig:ex1_error}, the errors for the two new approximants appear smaller within $K$, but larger outside $K$; this is to be expected, since the new sample points are closer to the origin.
}

\davide{
More quantitatively, we show in \cref{fig:ex1_poles_err_oth} the convergence of the pole approximation error with respect to $S$ (here as well, we keep $N=S-1$). For fast LS Pad\'e approximants, when $S$ is small, the error converges quite quickly (the error is uniformly smaller than the one obtained with samples at the roots of unity), but then it stagnates, and even increases starting from $S=20$; this is a symptom of the instability which affects the construction of this centered approximant, namely the computation of high-order derivatives of a meromorphic function.

The results for quasi-random sample points appear relatively similar, with the error flattening out starting from $S=20$.
In this case, the non-monotone behavior of the error for large $S$ is due to the interaction between two conflicting effects: the addition of sample points close to the resonances (which would, on its own, improve the accuracy of the approximation) and the ill-conditioning of the interpolation problem (for $S=30$, the Lebesgue constant is $\sim 10^7$, as opposed to $\sim 3$ for the $30$-th roots of unity).
}

\subsection{Time-harmonic vibrations of a tuning fork}\label{sec:scattering}
Let $\Omega\subset\R^3$ be the region occupied by a tuning fork. Assume that the tuning fork is clamped on a portion $\Gamma_D$ of the handle. Let $\smash{\widetilde{\gbf}}(t,\xx)=e^{\iota 2\pi\overline{\nu}t}\gbf(\xx)$ be a pressure pulse (sinusoidal in time and gaussian in space) impinging on a portion $\Gamma_{\gbf}$ of the exterior of the tuning fork, which we assume to be sound-hard. See \cref{fig:ex2_domain} for a graphical representation of the setup.

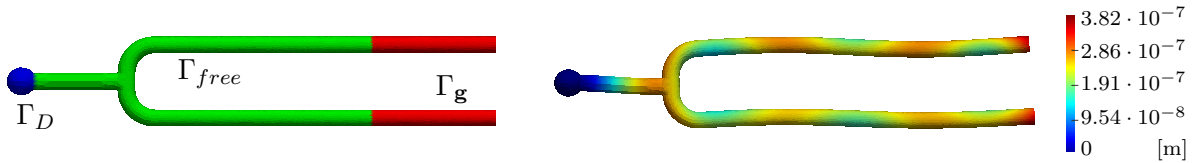
\begin{figure}[t]
\begin{center}
\hspace{-1.5cm}
\begin{overpic}[height = 1.8cm, trim = {1.25cm, 1.5cm, .9cm, 1.95cm}, clip]{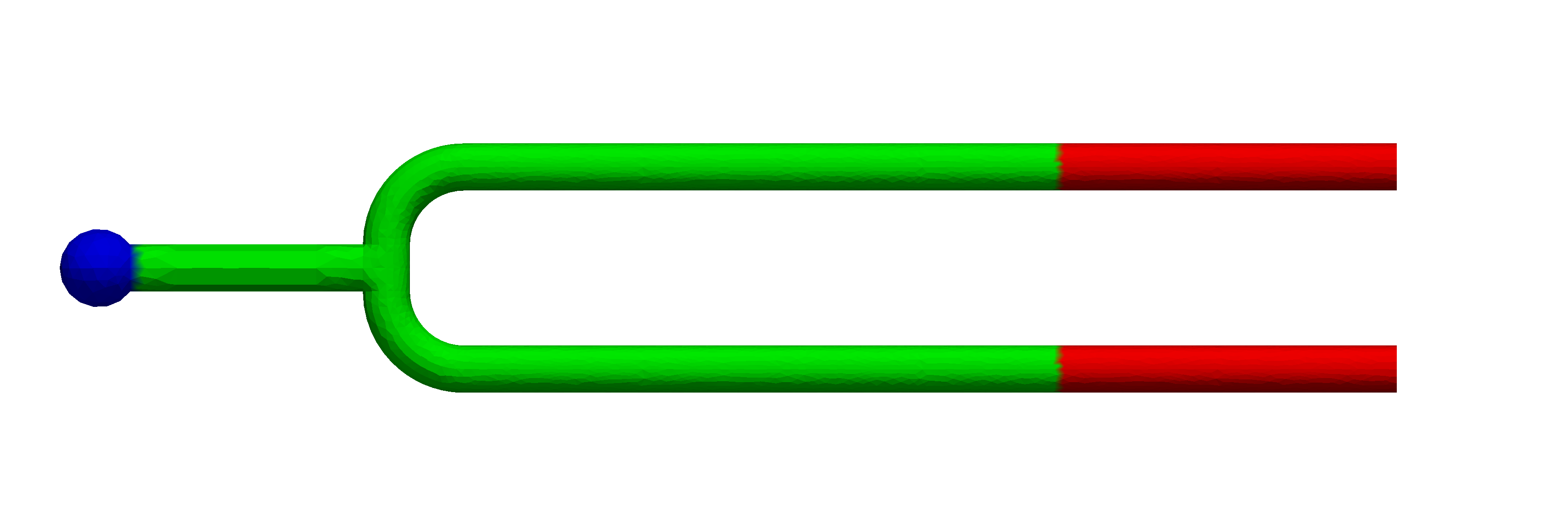}
 \put (2,5) {$\Gamma_D$}
 \put (32,13.5) {$\Gamma_{free}$}
 \put (80,10.5) {$\Gamma_{\gbf}$}
\end{overpic}
\begin{overpic}[height = 1.8cm, trim = {1.25cm, 1.5cm, .9cm, 1.75cm}, clip]{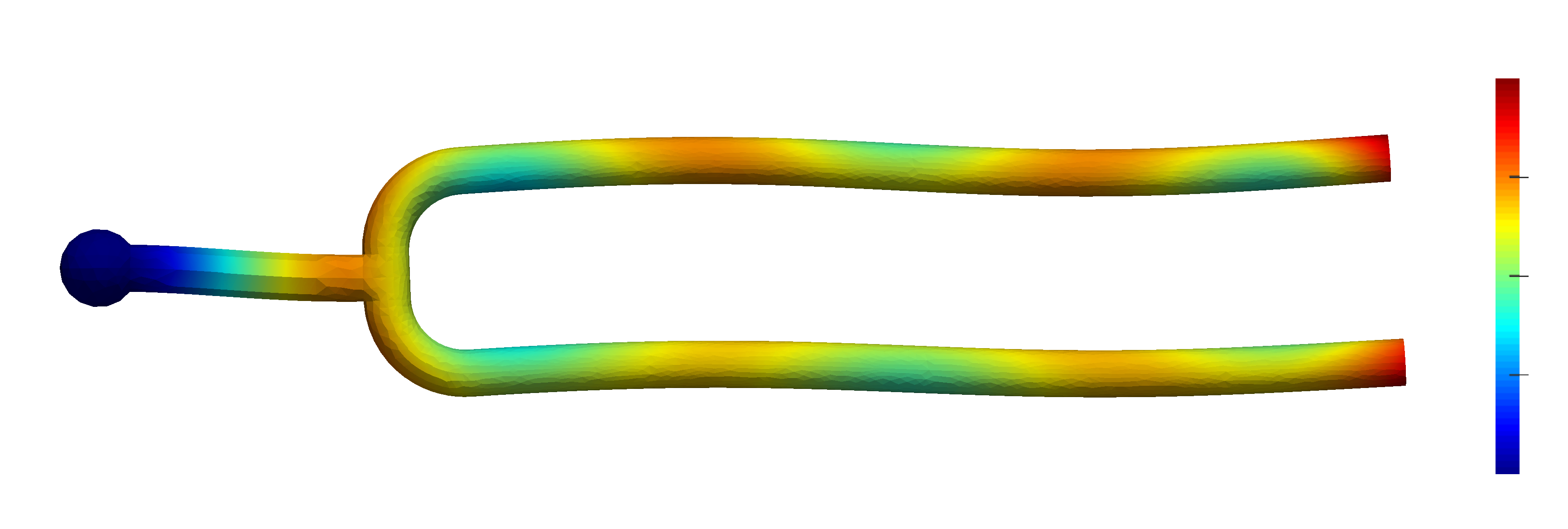}
 \put (100.5,-.75) {\scriptsize $0$\hspace{.85cm}[m]}
 \put (100.5,5.125) {\scriptsize $9.54\cdot 10^{-8}$}
 \put (100.5,11.75) {\scriptsize $1.91\cdot 10^{-7}$}
 \put (100.5,17.875) {\scriptsize $2.86\cdot 10^{-7}$}
 \put (100.5,24.25) {\scriptsize $3.82\cdot 10^{-7}$}
\end{overpic}
\caption{Reference domain (left) and domain warped by the real part of the displacement, magnified by a factor of $3\cdot 10^3$ (right). The displacement is the solution of \eqref{eq:helmholtz_fork} for frequency $\nu=10$ kHz and damping coefficient $\eta=0$ Hz. The total number of degrees of freedom of the discretized problem is 36330.}
\label{fig:ex2_domain}
\end{center}
\end{figure}

We are interested in computing the time-harmonic deformation $\uu(\nu)$ with frequency $\nu$ which the tuning fork undergoes. Assuming the tuning fork to have density $\rho$ and linear stress constitutive relation $\sigma(\cdot)$, the displacement $\uu(\nu)$ can be found by solving a (damped) time-harmonic linear elasticity problem
\begin{equation}\label{eq:helmholtz_fork}
\begin{cases}
-\nabla\cdot\sigma(\uu(\nu))-2\pi\nu(2\pi\nu-\eta\iota)\rho\uu(\nu)=\bm{0}\quad&\text{in }\Omega\\
\uu(\nu)=\bm{0}\quad&\text{on }\Gamma_D\\
\sigma(\uu(\nu))\nn=\gbf\quad&\text{on }\Gamma_{\gbf}\\
\sigma(\uu(\nu))\nn=\bm{0}\quad&\text{on }\Gamma_{free}=\partial\Omega\setminus\Gamma_D\setminus\Gamma_{\gbf}
\end{cases}\text{.}
\end{equation}
\davide{As $V$-inner product, we choose the elastic energy inner product:
\begin{equation*}
\dualV{\uu}{\vv}=\int_\Omega\sigma(\uu):\varepsilon(\vv)\textrm{d}\xx=\int_\Omega\sigma(\uu):\frac{\nabla\vv+\nabla\vv^\top}2\textrm{d}\xx.
\end{equation*}}

In particular, we choose to approximate problem \cref{eq:helmholtz_fork} using $\mathbb{P}_1$ FE on a sufficiently fine tetrahedral discretization of $\Omega$. The FE discretization of \cref{eq:helmholtz_fork} defines a parametric problem of the form \cref{eq:parametric_problem}, whose solution map $\uu$ can be proven meromorphic using compactness arguments \cite{Bonizzoni2018, Lenoir1992}. Thus, we decide to approximate it for $\nu\in K=[10,40]$ kHz with minimal rational interpolants relying on samples at the Chebyshev nodes of $K$. The chosen polynomial norm $\normT{\cdot}_N$ is the one induced by the family of Chebyshev polynomials over $K$, see \cref{sec:Convergence preliminaries}.

\begin{figure}[htp]
\begin{center}
\includegraphics[height = 5cm]{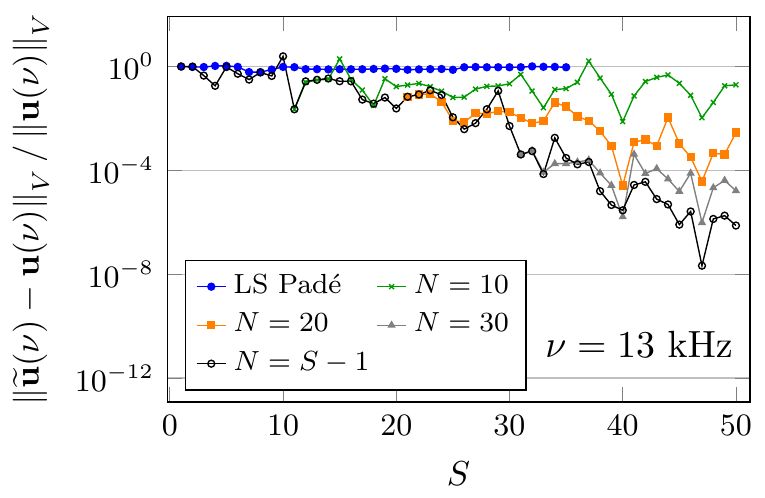}%
\hspace{-.3cm}
\includegraphics[height = 5cm]{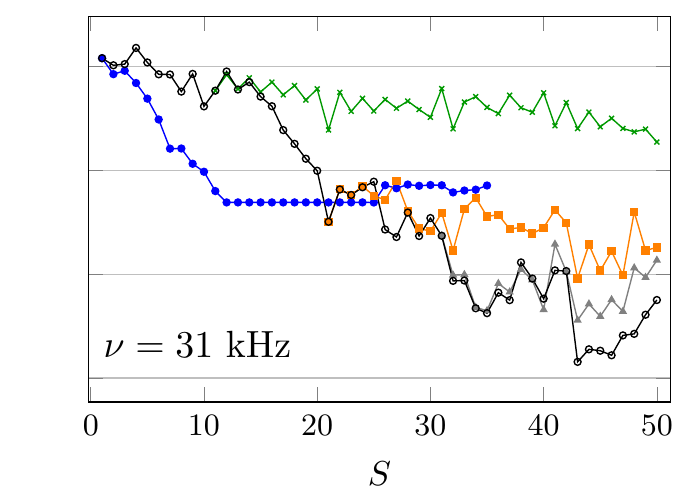}
\caption{Convergence of the approximation error in the energy norm with respect to the number of samples at $\nu=13$ kHz (left) and $\nu=31$ kHz (right).}
\label{fig:ex2_conv}

\vspace{.75cm}
\includegraphics[height = 5cm]{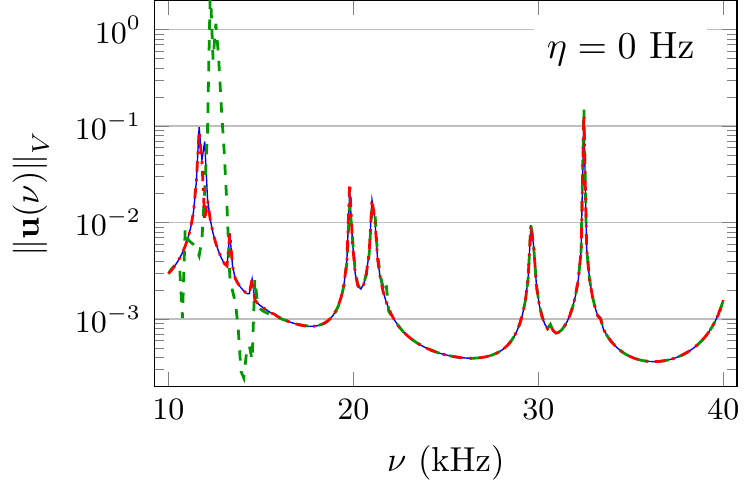}%
\hspace{.2cm}
\includegraphics[height = 5cm]{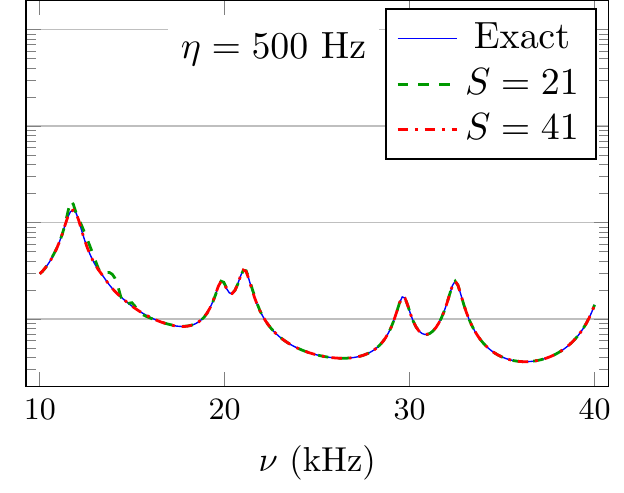}
\caption{Energy norm of the solution map $\uu$ of \eqref{eq:helmholtz_fork} in blue, for $\eta\in\{0,500\}$ Hz. In green and red the norm of the $[20/20]$ and $[40/40]$ minimal rational interpolant of $\uu$.}
\label{fig:ex2_norm}

\vspace{.75cm}
\includegraphics[height = 5cm]{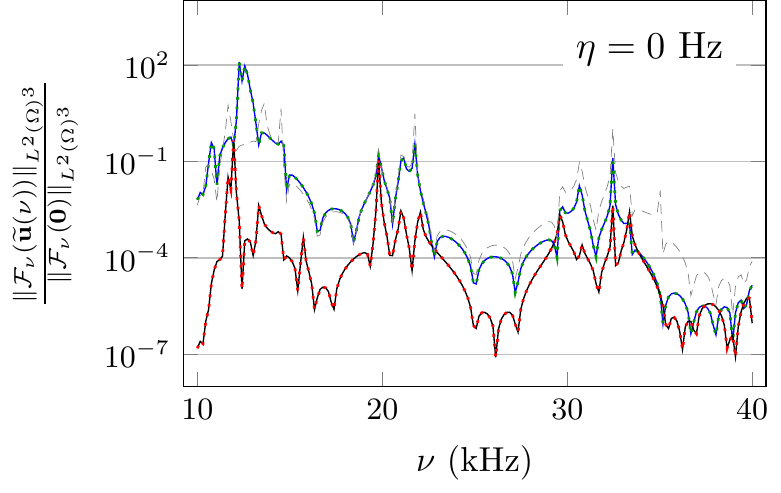}%
\hspace{.2cm}
\includegraphics[height = 5cm]{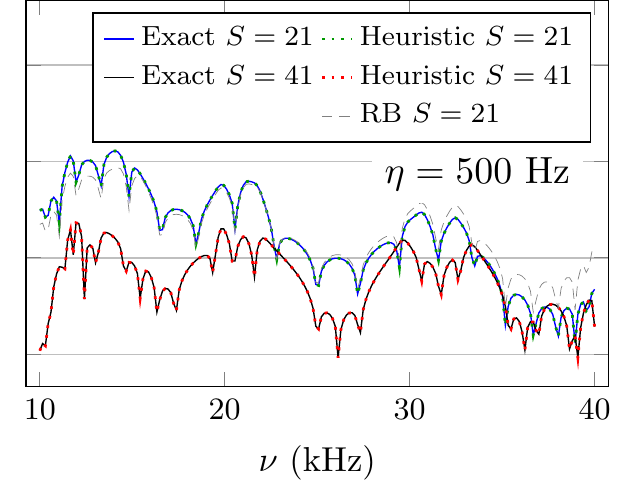}
\caption{Relative $L^2(\Omega)^3$-norm of the residual for problem \eqref{eq:helmholtz_fork} and $\widetilde{\uu}=\padegen{\uu}{20}{\Xi_{21}}$ (blue) and $\widetilde{\uu}=\padegen{\uu}{40}{\Xi_{41}}$ (black). In green and red the heuristic approximations \eqref{eq:residual_bound_linear} of the residual. The dashed grey line is the residual of the RB approximation of $\uu$ with samples at $\Xi_{21}$.}
\label{fig:ex2_res}
\end{center}
\end{figure}

\davide{First, we investigate the behavior of the relative approximation error $\frac{\normV{\widetilde{\uu}(\nu)-\uu(\nu)}}{\normV{\uu(\nu)}}$ at $\nu=13$ kHz and $\nu=31$ kHz, for $\eta=0$ Hz. We apply different approximation strategies: for $S=1,\ldots,51$, we consider $[S-1/S-1]$ fast LS Pad\'e approximants \cite{Bonizzoni2018b} with $S$ samples at the center of $K$, $[S-1/N]$ minimal rational approximants for fixed $N=10,20,30$, and diagonal $[S-1/S-1]$ minimal rational approximants.}

\davide{The results are shown in \cref{fig:ex2_conv}. 
As expected from \cref{th:error_N}, minimal rational approximants of type $[S-1/10]$ do not converge ($K$ contains 16 resonances), whereas the ones of type $[S-1/20]$ and $[S-1/30]$ do; we can observe that the diagonal approximant $[S-1/S-1]$ yields the best results overall.

The error achieved by fast LS Pad\'e approximants converges extremely quickly at $\nu=31$ kHz for small $S$, but no convergence can be observed at $\nu=13$ kHz; this is reasonable, since this type of approximant, relying on information at a single point, can deliver an extremely accurate surrogate close to the center, but the quality of the approximation degrades as we move further away. Also, we observe that the convergence plots for fast LS Pad\'e approximants stop decreasing already at $S=11$, and terminate abruptly at $S=35$: for $S\geq 36$, the construction of the surrogate model fails, due to conditioning issues in the evaluation of the derivatives of the solution map, and in the application of the (Taylor) interpolation operator \cite{Bonizzoni2018b,BonizzoniPAMM}.}

The norm of the solution map is shown in \cref{fig:ex2_norm} for $\eta\in\{0,500\}$ Hz, along with the approximated norm profiles, obtained by minimal rational interpolants of type $[20/20]$ and $[40/40]$. If $S=21$, for both values of $\eta$, we can observe that the approximation of $\normV{\uu}$ is quite accurate except for the low-frequency region, where the approximant for $\eta=0$ Hz actually appears quite unstable. This behavior can be improved by increasing the number of samples to $S=41$.

In a (more) realistic application, in order to determine whether there is actual need for more than 21 samples, one could rely on residual-based estimators such as \cref{eq:residual_tolerance}. To this aim, we can apply \cref{eq:affine_residual_decomposition} to obtain an expression for the residual in the $L^2(\Omega)^3$-norm. The evaluation of \cref{eq:affine_residual_decomposition} on a fine grid is shown in \cref{fig:ex2_res}.

In the same Figure, we also show the values obtained through the simplified residual estimator \cref{eq:residual_bound_linear_ineq}. In particular, in order to evaluate the estimator, we need to compute or approximate the operator norm $\norm{F_1}{\mcL(V;L^2(\Omega)^3)}$, with $F_1$ the derivative of the operator in \cref{eq:helmholtz_fork} at some point $\mu'$, see \cref{sec:Greedy}. \davide{We manage to avoid this, thus keeping the estimator as non-intrusive as possible, by the following argument: if we assume that \cref{eq:residual_bound_linear} holds approximately (the problem is quadratic and not linear in the parameter), we can identify a constant $C$, independent of $\mu$,} for which, employing the notation of \cref{sec:Greedy},
\begin{equation*}
\normW{\mcF_\mu\left(\padet(\mu)\right)}\approx C\abs{\frac{\omega^{Z_S}(\mu)}{\dent(\mu)}}\text{.}
\end{equation*}
\davide{In practice, the value of $C$ can be easily found} by evaluating (offline and at a relatively low computational cost) the exact value of the residual at some point $\mu'\in K\setminus Z_S$, so that the value of $C$ can be set to
\begin{equation*}
C=\normW{\mcF_\mu\left(\padet(\mu')\right)}\abs{\frac{\dent(\mu')}{\omega^{Z_S}(\mu')}}\text{.}
\end{equation*}
In our example we \davide{choose} $\mu'=25.5$ kHz, and we can observe the corresponding residual estimator to be extremely accurate. Indeed, the values obtained with the two estimators never differ by more than 1\%, and the corresponding curves cannot be distinguished in the plot.

Finally, in \cref{fig:ex2_res} we also perform a comparison between minimal rational interpolants and RBs, by showing the norm of the residual obtained with the RB approximant of \cref{eq:helmholtz_fork} computed from samples at $\Xi_{21}$. We can observe that the two residuals are quite similar in terms of both order of magnitude and behavior with respect to $\nu$. In particular, simplifying grossly, minimal rational interpolants seem to have a very slight edge on RBs at higher frequencies, while the opposite appears to be true at lower frequencies.

\section{Conclusions}\label{sec:conclusions}
In this paper we have presented minimal rational interpolants as the natural extension of fast LS-Pad\'e approximants \cite{Bonizzoni2018b}, and we have shown that most of the theoretical properties of this latter technique generalize nicely to the former. The new method can be considered superior to the original one for two main reasons:
\renewcommand{\labelenumi}{\textbullet}
\begin{enumerate}
\item the stability is greatly improved, since computing derivatives of the solution map accumulates numerical errors; the Arnoldi-like reorthogonalization technique described in \cite{Bonizzoni2018b} can somehow limit but not totally eliminate this issue;
\item the distributed sampling helps to achieve a globally (e.g. using the $L^\infty(K)$-metric) small approximation error, which is the target in most applications; of course, an interpolatory approach cannot be as accurate close to the samples as a Taylor-type approximation, assuming the number of samples to be the same.
\end{enumerate}

Our numerical example has shown minimal rational interpolants to preform well in the MOR \davide{of a normal eigenproblem and} of a time-harmonic problem. In particular, we have shown the effectiveness of a heuristic and cheap to compute \emph{a posteriori} residual estimator. Accordingly, we deem of interest to investigate a possible greedy-type algorithm, in the same spirit as the one for the RB method, but with minimal rational interpolants replacing Galerkin projection. In particular, it is still unclear \davide{how well} the approximation properties of the method are preserved if the samples are not computed at \davide{optimal (Fej\'er)} points.

\subsection*{Acknowledgments}
The author gratefully thanks Fabio Nobile for the many fruitful discussions and for his helpful advice at several stages of this work.

\appendix
\section{Polynomial norm bounds}\label{ap:with_scale}
Let $\Gamma\subset\C$ be a Jordan curve, and consider $\{\varphi_j\}_{j=0,1,\ldots}$ a hierarchical (degree$(\varphi_j)=j$ for all $j$) family of polynomials, orthonormal over $L^2(\Gamma)$:
\begin{equation}
\frac{\int_{\Gamma}\varphi_j\overline{\varphi_i}\abs{\text{d}l}}{\int_{\Gamma}\abs{\text{d}l}}=\delta_{ij}\quad\forall(i,j)\in\{0,1,\ldots\}^2\text{.}
\end{equation}
\davide{For instance, if $\Gamma=\partial\ball{\mu_0}{1}$, we can choose $\varphi_j(\mu)=\mu^j$ for all $j$.}

This, for all $N\in\N$, induces a norm over $\Pspace{N}$ in the following sense:
\begin{equation}\label{eq:hierarchical_norm}
\normT{Q}_N=\sqrt{\sum_{j=0}^N\abs{q_j}^2},\quad\text{where}\quad Q=\sum_{j=0}^Nq_j\varphi_j.
\end{equation}
The following holds.
\begin{lemma}
The norm $\normT{\cdot}_N$ in \cref{eq:hierarchical_norm} satisfies \cref{as:with_scale} for any $\mu_0$ chosen within the interior of $\Gamma$.
\end{lemma}

\begin{proof}
For a given $N\in\N$, let $0\leq D\leq N$, $\{z_j\}_{j=1}^D\in\C^D$, and $Q=\prod_{j=1}^D\left(\,\cdot-z_j\right)$. Due to orthonormality, it holds
\begin{equation}
\normT{Q}_N^2\int_\Gamma\abs{\text{d}l}=\int_\Gamma\abs{Q}^2\abs{\text{d}l}=\int_\Gamma\prod_{j=1}^D\abs{\,\cdot-z_j}^2\abs{\text{d}l}\text{.}
\end{equation}

On one side, let $R_{\mu_0}=\max_\Gamma\abs{\,\cdot-\mu_0}$. We can use the triangular inequality to obtain
\begin{align}
\normT{Q}_N^2\int_\Gamma\abs{\text{d}l}\leq &\int_\Gamma\prod_{j=1}^D\left(\abs{\,\cdot-\mu_0}+\abs{\mu_0-z_j}\right)^2\abs{\text{d}l}\nonumber\\
\leq&\prod_{j=1}^D\left(R_{\mu_0}+\abs{\mu_0-z_j}\right)^2\int_\Gamma\abs{\text{d}l}\text{.}\label{eq:with_scaleR}
\end{align}
On the other hand, let $c_{\mu_0}=\left(2\pi\min_\Gamma\abs{\,\cdot-\mu_0}/\int_\Gamma\abs{\textup{d}l}\right)^2$. We can exploit the Cauchy-Schwarz inequality to show that
\begin{equation*}
\left(\int_\Gamma\prod_{j=1}^D\abs{\,\cdot-z_j}^2\abs{\text{d}l}\right)\left(\int_\Gamma\frac{\abs{\text{d}l}}{\abs{\,\cdot-\mu_0}^2}\right)\geq\abs{\int_\Gamma\frac{\prod_{j=1}^D\left(\,\cdot-z_j\right)}{\,\cdot-\mu_0}\text{d}l}^2\text{,}
\end{equation*}
so that, due to the Cauchy integral formula,
\begin{align}
\normT{Q}_N^2\int_\Gamma\abs{\text{d}l}\geq &\abs{\int_\Gamma\frac{\prod_{j=1}^D\left(\,\cdot-z_j\right)}{\,\cdot-\mu_0}\text{d}l}^2\left(\int_\Gamma\frac{\abs{\text{d}l}}{\abs{\,\cdot-\mu_0}^2}\right)^{-1}\nonumber\\
\geq&4\pi^2\prod_{j=1}^D\abs{\mu_0-z_j}^2\left(\min_\Gamma\abs{\,\cdot-\mu_0}\right)^2\left(\int_\Gamma\abs{\text{d}l}\right)^{-1}\nonumber\\
\geq&c_{\mu_0}\prod_{j=1}^D\abs{\mu_0-z_j}^2\int_\Gamma\abs{\text{d}l}\text{.}\label{eq:with_scaleL}
\end{align}
Dividing \cref{eq:with_scaleR} and \cref{eq:with_scaleL} by $\int_\Gamma\abs{\text{d}l}$ yields the proof.
\end{proof}

\bibliographystyle{plain}
\bibliography{RMOR_bib}
\end{document}